\documentclass[12pt,a4paper]{amsart}
\usepackage{amsfonts,xcolor}
\usepackage{amsthm}
\usepackage{amsmath}
\usepackage{amscd,breqn}
\usepackage[latin2]{inputenc}
\usepackage{t1enc}
\usepackage[mathscr]{eucal}
\usepackage{indentfirst}
\usepackage{graphicx}
\usepackage{graphics}
\usepackage{pict2e}
\usepackage{epic}
\numberwithin{equation}{section}
\usepackage[margin=2.9cm]{geometry}
\usepackage{epstopdf}

 \usepackage{pgfplots}

\usepackage{tikz}
\usetikzlibrary{decorations.pathreplacing}
\usepackage{amsfonts,latexsym,rawfonts,amsmath,amssymb,amsthm, mathrsfs, lscape}

\newtheorem{theorem}{Theorem}[section]
\newtheorem{lemma}[theorem]{Lemma}
\newtheorem{proposition}[theorem]{Proposition}

\theoremstyle{definition}
\newtheorem{definition}[theorem]{Definition}

\theoremstyle{remark}
\newtheorem{remark}[theorem]{Remark}
\numberwithin{equation}{section}

\begin{document}


\title[Degenerations of Negative K\"ahler-Einstein Surfaces]{Degenerations of Negative K\"ahler-Einstein Surfaces}  

\author[Holly Mandel]{Holly Mandel}
\address[HM]{Department of Mathematics, University of California, Berkeley, \linebreak Berkeley, CA 94720}
\email{holly.mandel@berkeley.edu}

\date{\today}


\keywords{K\"ahler geometry, K\"ahler-Einstein metrics}

\begin{abstract}
Every compact K\"ahler manifold with negative first Chern class admits a unique metric $g$ such that $\text{Ric}(g) = -g$. Understanding how families of these metrics degenerate gives insight into their geometry and is important for understanding the compactification of the moduli space of negative K\"ahler-Einstein metrics. I study a special class of such families in complex dimension two. Following the work of Sun and Zhang (2019) in the Calabi-Yau case, I construct a K\"ahler-Einstein neck region interpolating between canonical metrics on components of the central fiber. This  provides a model for the limiting geometry of metrics in the family.
\end{abstract}

\maketitle

\section{Introduction}
\subsection{K\"ahler-Einstein metrics}
K\"ahler-Einstein metrics sit at the intersection of physics, differential geometry, and algebraic geometry. In physics, they form a class of solutions to the Einstein field equations.  In differential geometry, they are higher-dimensional analogues of constant curvature metrics on Riemann surfaces that are more rigid than constant scalar curvature metrics and less rigid than constant curvature metrics. In algebraic geometry, they are canonical objects associated to certain complex varieties whose properties reflect the underlying algebraic structure.

The existence theory of K\"ahler-Einstein metrics on compact manifolds dates back to Yau's theorem \cite{Yau} but was partially open until 2014. Since the Ricci curvature of a manifold $X$ represents its first Chern class, the existence of a K\"ahler-Einstein metric on $X$ implies that $c_1(X)$ has a representative that is either positive definite, negative definite, or identically zero. Conversely, if $c_1(X) = 0$, the existence of a Ricci flat metric in any K\"ahler class of $X$ follows from Yau's theorem, while if $c_1(X)$ has a negative representative, the existence of a unique (up to rescaling) K\"ahler-Einstein metric with negative scalar curvature was proved by Aubin and Yau in 1978 \cite{Aubin, Yau}. If $c_1(X)$ is positive, however, there are nontrivial obstructions to existence \cite{Fut,Lich,Mat}. A full understanding was not achieved until 2014, when Chen-Donaldson-Sun proved that existence is equivalent to the algebraic condition of K-stability \cite{CDS1,CDS2,CDS3}.

In this paper I study the case $c_1(X) < 0$. I call a K\"ahler-Einstein metric on such a space a \textit{negative K\"ahler-Einstein metric}, since the constant of proportionality between the metric and its Ricci curvature is negative. Though existence and uniqueness in this case have long been established, the proof is implicit and provides little geometric information. It remains a challenge in K\"ahler geometry to characterize these metrics.

One goal is to describe the compactification of the moduli space of negative K\"ahler-Einstein metrics. Since limits in this space will not always be smooth, a key question is how families of such metrics degenerate. For instance, given a manifold $X$, we can vary the complex structure on $X$ to produce a family of K\"ahler-Einstein metrics. If the complex manifold develops a singularity, the limit space does not have a K\"ahler-Einstein metric in the usual sense, but we can try to define a generalization to compactify the family.

A first step in this direction is to characterize the convergence of the family outside a singular set. Previous work has explicitly demonstrated convergence to known metrics. If $X = \bar{X} \setminus D$ for a projective manifold $\bar{X}$ and smooth divisor $D$ such that $c_1(X) < 0$, there exists a unique complete, finite-volume negative K\"ahler-Einstein metric on $X$ \cite{Bando, CY, Kobayashi, TY2}. Tian \cite{Tian} showed that a degenerating family of K\"ahler-Einstein metrics will Gromov-Hausdorff converge to this metric on the smooth locus of the central fiber under the assumptions that the total space of the degeneration is smooth, the central fiber has only normal crossing  singularities, and its components intersect only pairwise. The pairwise intersection assumption was later removed by Ruan \cite{Ruan}. Greater generality was achieved by Song \cite{Song}, who used results from birational geometry to obtain convergence without loss of volume for a general algebraic degeneration and to further characterize the structure of the central fiber. 


I aim to investigate the geometry that collapses to the singular set. This perspective will be necessary for understanding what types of spaces are needed to compactify the moduli space of negative K\"ahler-Einstein metrics. In addition, the techniques I use give an explicit description of the nonsingular metrics close to the central fiber. This allows us to ``see'' the negative K\"ahler-Einstein metrics whose existence has long been established but whose geometry is mostly unknown. 

\subsection{Degenerations of negative K\"ahler-Einstein metrics}
A degeneration of negative K\"ahler-Einstein metrics is defined as a flat family $\pi: \mathcal{X} \rightarrow \Delta$ of algebraic varities over the complex disc $\Delta$ such that $X_t = \pi^{-1}(t)$ is smooth for $t \neq 0$ and $K_{\mathcal{X}/\Delta}$ is positive. For generic $t \in \Delta$, $K_{\mathcal{X}/\Delta}\vert_{X_t} \simeq K_{X_t}$, so there is a unique K\"ahler-Einstein metric $\omega_t$ in $2\pi \, c_1(K_{\mathcal{X}/\Delta}\vert_{X_t})$. 

In this paper we investigate a specific family of the above type. For $i = 1, 2, 3$, let $f_i$ be a homogeneous polynomial in $4$ variables of degree $d_i$ such that $d_1 + d_2 = d_3 > 4$. Let $\mathcal{X} \subseteq \mathbb{C}P^3 \times \Delta$ be the variety \[ X_t = V(f_1 f_2 - t f_3),\] where the $f_i$ are interpreted as polynomials on $\mathbb{C}P^3$ and $t$ is the coordinate on $\Delta$. Say that $Y_i = V(f_i)$ is smooth for $i = 1,2$ and $X_t$ is smooth for $t \neq 0$. Finally, assume that $D = V(f_1) \cap V(f_2)$ and $D \cap V(f_3)$ are complete intersections. By the adjunction formula, $K_{X_t}$ is ample for generic $t$. 

Sun and Zhang \cite{sz19} have characterized the Calabi-Yau case ($d_1 + d_2 = 4$). They found that after rescaling to unit diameter, $X_t$ converges in the Gromov-Hausdorff topology to an interval in $\mathbb{R}$. The interior of the interval reflects the geometry of an infinitesimal neighborhood of the singular point but contains all of the rescaled volume of the space. Meanwhile, rescaled limits at the end points converge to the complete Calabi-Yau metrics on $Y_i \setminus D$ constructed by Tian and Yau \cite{TY}.  An interesting corollary is that these Tian-Yau metrics, though not known to be unique as solutions to a prescribed Ricci curvature problem, are canonical in the sense they arise from degenerating families. 
 
In our case, $Y_i \setminus D$ admits a unique complete K\"ahler-Einstein metric \cite{Bando, CY, Kobayashi, TY2}. We hypothesize that as in \cite{sz19}, $(X_t,\omega_t)$ will degenerate to a space with three parts: one component for $Y_i$, $i=1,2$, equipped with the K\"ahler-Einstein metric on $Y_i \setminus D$, plus a neck region gluing these spaces near $D$. Like the Tian-Yau metric, the K\"ahler-Einstein metric on  $Y_i \setminus D$ resembles a Calabi model space over the normal bundle of $D$ in $Y_i$ near $D$, so a suitable neck region would be a K\"ahler-Einstein interpolation between these two Calabi model spaces. Indeed the bulk of \cite{sz19} is the construction of an analogous neck region in the Calabi-Yau case. 

The main result of this paper is the succesful construction of this neck region. We will give a precise definition of $(\mathcal{C}_{\pm},g_{\mathcal{C}_{\pm}})$ in Section \ref{calabimodel}.

\begin{theorem}\label{mainthm}
Fix a complex curve $D$ with $c_1(D) < 0$ and integers $k_- \geq 0$, $k_+ \leq 0$. Let $(\mathcal{C}_{\pm},g_{\mathcal{C}_{\pm}})$ be the Calabi model space over $k_- \mathcal{N}_D$ and $-k_+ \mathcal{N}_D$, respectively. Then there exists $\alpha_0 \in (0,1)$, a manifold $\mathcal{M}$ with boundary components $\partial \mathcal{M}_\pm$ that gives a singular $S^1$ fibration over $D \times I$ for some interval $I \subset \mathbb{R}$, and a family of $S^1$-invariant negative K\"ahler-Einstein metrics $\omega_{\text{KE},T}$ on $\mathcal{M}$, such that the following holds: for any $\alpha \in (0,\alpha_0)$, $\epsilon > 0$, $k \in \mathbb{Z}_{\geq 0}$, and $R > 0$, $B_R(\partial \mathcal{M}_{-})$ is $\epsilon$-close in $C^{k,\alpha}$ to a ball in $\mathcal{C}_-$, and similarly for $\partial \mathcal{M}_{+}$, for $T \gg 0$. 
\end{theorem}  
Our proof yields a detailed description of the geometry of $(\mathcal{M},\omega_{\text{KE},T})$ for large $T$.  The diameter of $(\mathcal{M},\omega_{\text{KE},T})$ grows without bound as $T \rightarrow \infty$, and if we rescale the diameter to a constant, the resulting spaces collapse to an interval in $\mathbb{R}$. We can say more, however, about the pointed convergence of $(\mathcal{M},\omega_{\text{KE},T})$. 
\begin{theorem}\label{submainthm}
Under the assumptions of Theorem \ref{mainthm}, there exists a family of functions $W_T: \mathcal{M} \rightarrow \mathbb{R}_{>0}$ with the following property: Let $(x_j)_{j=1}^{\infty}$ be a sequence of points in $\mathcal{M}$ and choose a sequence $T_j \rightarrow \infty$. Then there exists a subsequence of $(\mathcal{M},W_{T_j}(x_j)^{-2}\omega_{\text{KE},T_j},x_j)$ that converges in the pointed Gromov-Hausdorff topology to one of the following:
\begin{enumerate}
\item the Taub-NUT space $(\mathbb{C}^2_{TN},g_{TN})$,
\item the Riemannian product $\mathbb{C} \times \mathbb{R}$,
\item the Riemannian cylinder $D \times \mathbb{R}$,
\item the Calabi model space $(\mathcal{C}_{\pm},g_{\mathcal{C}_{\pm}})$. 
\end{enumerate}
In cases 1 and 4, convergence is smooth. In cases 2 and 3, there is collapsing with bounded curvature away from finitely many points.
\end{theorem}
The rescaling factor $W_T$ is related to the local regularity scales of $\omega_{\text{KE},T}$. In cases $2$ and $3$, there is smooth convergence without collapsing on local universal covers. Thus the theorem provides an explicit pointwise description of $\omega_{\text{KE},T}$ up to error terms that decay as $T \rightarrow \infty$. 

To construct $(\mathcal{M},\omega_{\text{KE},T})$, we guess that the desired metric can be approximated by a K\"ahler metric with $S^1$ symmetry on a singular $S^1$ fibration over $D \times I$ for some interval $I \subset \mathbb{R}$. The constraints on the end behavior of the metric determine the topology of $\mathcal{M}$. The K\"ahler-Einstein equation on $\mathcal{M}$ can then be expressed in terms of the fiber size $h^{-1}$ and a scaling $\chi$ of a fixed metric on $D$ (Section \ref{kahler_reduction}). We solve the linearization of this reduced equation to construct a family of approximately K\"ahler-Einstein metrics $\omega_T$ (Sections \ref{solution} and \ref{m_construction}). By adding an inhomogeneous delta function term to the linearized equation, we change the degree of the restriction of the $S^1$ bundle to $D$ to match the Calabi model spaces at either end of $\mathcal{M}$. Once we have solved the inhomogeneous linearized equation, we investigate the local rescaled geometry of $\omega_T$ as $T \rightarrow \infty$ (Section \ref{schauderestimatesection}).  The resulting description, reflected in Theorem \ref{submainthm}, allows us to derive weighted Schauder estimates that are independent of $T$. Finally, we use these estimates to correct $\omega_T$ to a K\"ahler-Einstein metric $\omega_{\text{KE},T}$ by the implicit function theorem (Section \ref{perturbation}). 

Our techniques follow \cite{sz19}, but major differences from the Calabi-Yau case appear in Sections \ref{solution}, \ref{schauderestimatesection}, and \ref{perturbation}. These differences result from the fact that the linearization of the negative K\"ahler-Einstein equation is $\Delta-1$ rather than $\Delta$. In Section \ref{solution}, this implies that the linearization of the K\"ahler-Einstein equation results (after separation of variables) in an ordinary differential equation whose solutions are qualitatively different from the exponential functions in \cite{sz19}. In Section \ref{perturbation}, this requires us to adopt a different framework for our use of the implicit function theorem. As a result, we must derive different Schauder estimates in Section \ref{schauderestimatesection}.

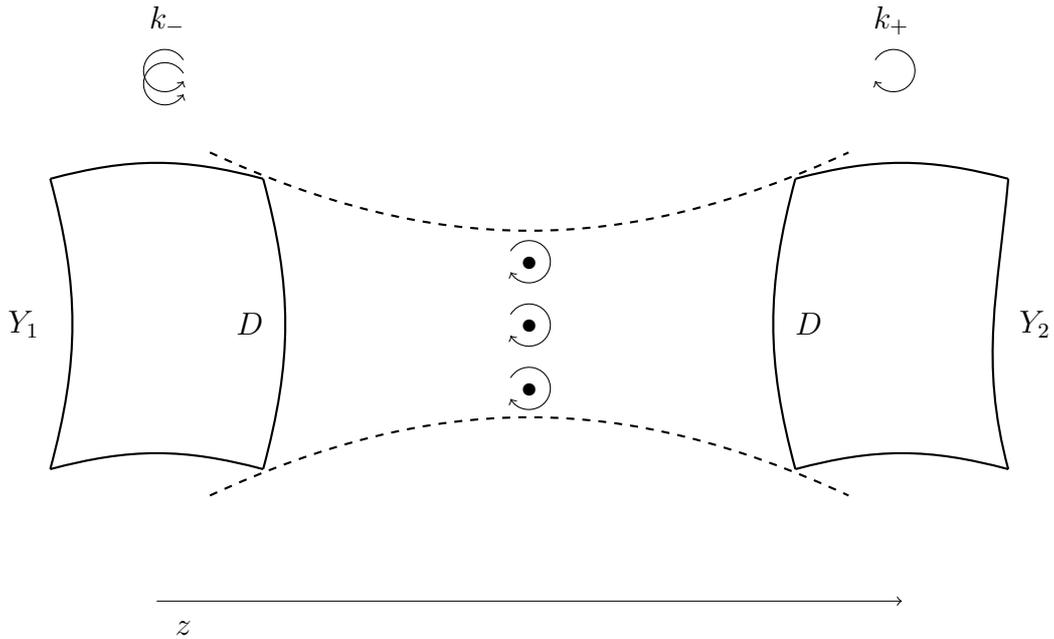
\begin{figure}[h]
\begin{center}
\begin{tikzpicture}[scale=.7]

\draw[thick] (-9, 5.5) to [out = 15, in = 165](-5, 5.5); 
\draw[thick] (-9, 0) to [out = 15, in = 165](-5, 0);
\draw[thick] (-9, 5.5) to [out = 285, in = 75](-9,0);  
\draw[black,  thick] (-5, 5.5) to [out = 285, in = 75](-5, 0); 
\node[black] at (-5.25,2.75) {$D$};
\node at (-9.5,2.75) {$Y_1$};

\draw[thick, dashed] (-6,6)  to [out = 335, in = 205] (6,6);  
\draw[thick, dashed] (-6,-0.5)  to [out = 25, in = 155] (6,-0.5); 

\draw[thick] (9, 5.5) to [out = 165, in = 15](5, 5.5); 
\draw[thick] (9, 0) to [out = 165, in = 15](5, 0);
\draw[thick] (9, 5.5) to [out = 265, in = 105](9, 0);  
\draw[black,  thick] (5, 5.5) to [out = 255, in = 105](5, 0); 
\node[black] at (5.25,2.75) {$D$};
\node at (9.5,2.75) {$Y_2$};

\node[black] at (0, 3.9) {$\bullet$};
\draw[black, <-] (-.35,3.73) arc (-150:150:.4);

\node[black] at (0, 2.7) {$\bullet$};
\draw[black, <-] (-.35,2.53) arc (-150:150:.4);

\node[black] at (0, 1.5) {$\bullet$};
\draw[black, <-] (-.35,1.33) arc (-150:150:.4);

\draw[->] (-6.5,7.75) arc (30:330:.4);
\draw[->] (-6.5,7.5) arc (30:330:.4);

\draw[<-] (6.5,7.35) arc (-150:150:.4);


%

\node at (-6.8,8.5) {$k_-$};
\node at (6.8,8.5) {$k_+$};

\draw[->] (-7,-2.5) to (7,-2.5);
\node at (-6.5,-3) {$z$};

\end{tikzpicture}
\end{center}
\caption[The construction of $(\mathcal{M},\omega_{\text{KE},T})$]{The construction of $(\mathcal{M},\omega_{\text{KE},T})$. The base space is the manifold $D \times I_z$. The diameter of $D$ is bounded below independently of $T$ and blows up near the singular points. $\mathcal{M}$ is a singular $S^1$ fibration over $D \times I$. The size of the $S^1$ fiber, given by $h^{-1}$, decreases to $0$ near the singular points. The degree of the restriction of the $S^1$ fibration to $D$ is given by $k_-$ for $z < 0$ and $-k_+$ for $z > 0$ and changes at $z = 0$ because there are $k_--k_+$ singular points.}
 \label{f: construction}
\end{figure}

\section{K\"ahler Reduction}\label{kahler_reduction}
\subsection{K\"ahler metrics with Hamiltonian symmetry} Our first step toward the construction of $(\mathcal{M},\omega_{\text{KE},T})$ is to create an approximately K\"ahler-Einstein space with the desired end behavior. We attempt to build such a space under the added assumption of $S^1$ symmetry. In the Calabi-Yau case this is the Gibbons-Hawkings ansatz. The discussion in this section is based on Section 2.1 in \cite{sz19} .

First, let's assume we already have such a space and see how our calculations are simplified. Let $(X,\omega,J)$ be a K\"ahler manifold of complex dimension $n$. We say that $X$ has a holomorphic $S^1$ symmetry if there is an action $\varphi: S^1 \times X \rightarrow X$ such that  $\varphi_\theta$ is holomorphic and $\varphi_\theta^{\ast}\omega = \omega$ for each $\theta \in S^1$, where $\varphi_{\theta}(x) = \varphi(\theta,x)$. In this case let $\xi$ be the vector field generating the action, so 
\[ \frac{\partial}{\partial \theta}\varphi(\theta,x) \bigg\vert_{\theta = \theta_0} = \xi(x).\] We say that $X$ is Hamiltonian if there exists a function $z$ such that 
\[dz = i_{\xi} \omega.\] 

If the holomorphic Hamiltonian $S^1$ action on $X$ is also free, we can simplify the description of $\omega$ and $J$ by dividing out the $S^1$ symmetry. Locally we can quotient by the orbits of the complexified action, generated by $\xi^{1,0} = \xi - i J \xi$, to form an $(n-1)$ dimension complex manifold $D$. Then we can identify $X$ with an $S^1$ bundle over $D \times I$ for some interval $I \subset \mathbb{R}$ with coordinate $z$. Since $\omega$ is $S^1$ invariant it can be parameterized by the base space $D \times I$. 

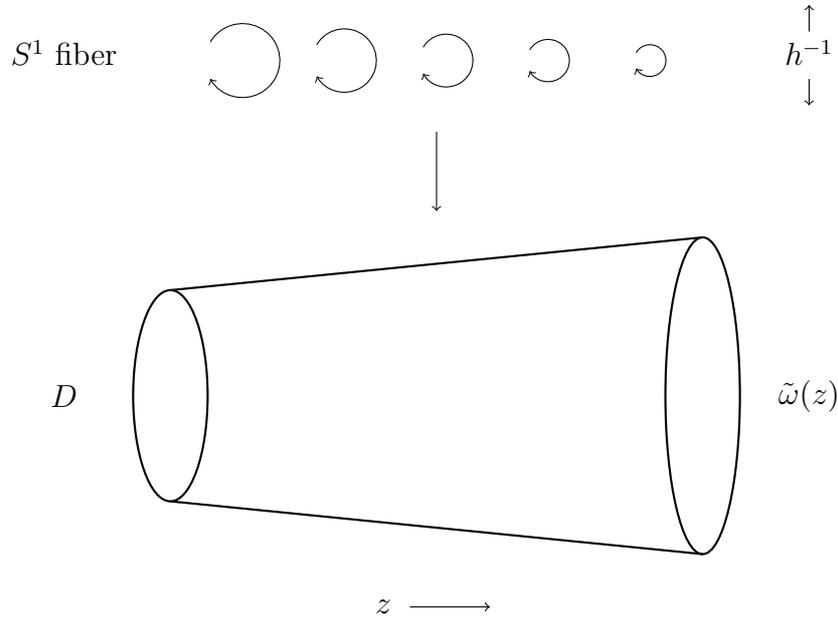
\begin{figure}[h]
\begin{center}
\begin{tikzpicture}[scale=0.7]
\draw[thick] (-5,0) ellipse (.7cm and 2cm);
\draw[thick] (5,0) ellipse (.7cm and 3cm);
\draw[thick] (-5,2) to (5,3);
\draw[thick] (-5,-2) to (5,-3);

\draw[<-] (-4.25,6) arc (-150:150:.7);
\draw[<-] (-2.25,6.05) arc (-150:150:.6);
\draw[<-] (-0.25,6.1) arc (-150:150:.5);
\draw[<-] (1.75,6.15) arc (-150:150:.4);
\draw[<-] (3.75,6.2) arc (-150:150:.3);

\draw[->] (0,5) to (0,3.5);

\node[black] at (-7,0){$D$};
\node[black] at (-7,6.5){$S^1$ fiber};

\node[black] at (7,6.5){$h^{-1}$};
\draw[->] (7,6.9) to (7,7.4);
\draw[->] (7,6.0) to (7,5.5);

\node[black] at (7,0){$\tilde{\omega}(z)$};

\node[black] at (-1,-4){$z$};
\draw[->] (-.5,-4) to (1,-4);
\end{tikzpicture}
\end{center}
\caption[K\"ahler reduction]{K\"ahler reduction. Locally, the total space is decomposed into an $S^1$ fibration over $D \times I_z$. The $S^1$-invariant metric on $X$ is specified by the fiber size, given by $h^{-1}$, and the $z$-family $\tilde{\omega}(z)$ of metrics on $D$.}
 \label{f: construction}
\end{figure}

Let $y$ be a local holomorphic coordinate on $D$ with the convention that $J dy = -i \, dy$, $J d\bar{y} = i \, d\bar{y}$. Let $t$ be a function such that $\xi(t) = 1$. Then $y,\overline{y},t,z$ form a local coordinate system for $X$, but the coordinates $z$ and $t$ are not holomorphic. We write
\[ J dz = h^{-1} (-dt + \theta)\]
where $h$ is a function and $\theta$ is an $S^1$-invariant one-form without a $dt$ component. We can determine $h$ by observing that by our choice of $z$,
\[ J dz(\partial_t) = - dz(J \partial_t) = - \omega(\xi, J \xi) = - \Vert \xi \Vert^{2}.\] 
Therefore $h =  \Vert \xi \Vert^{-2}$. 

Write $\Theta = -dt + \theta$ and define \begin{equation}\label{omega_defn}
\tilde{\omega} = \omega - dz \wedge \Theta.
\end{equation} The form $\tilde{\omega}$ does not have components in $dz$ or $dt$.  In addition,
\[ \mathcal{L}_{\xi} \tilde{\omega} = \mathcal{L}_{\xi} \omega - \mathcal{L}_{\xi} (dz \wedge \Theta). \]
The first term on the right vanishes due to $S^1$ invariance, and Cartan's formula can be used to show that $\mathcal{L}_{\xi} (dz \wedge \Theta) =  0$ as well. Thus $\tilde{\omega}$ can be thought of as a $z$-family of $(1,1)$-forms on $D$.

The integrability of $J$ and the K\"ahler condition on $\omega$ imply (see \cite{sz19} 2.1) that $h$ and $\tilde{\omega}$ satisfy the system
\begin{equation}\label{kahlerCondition} \bigg\lbrace \begin{array}{l} \partial_z^2\tilde{\omega} + d_D d_D^c h  =  0 \\ \partial_z \tilde{\omega} - dz \wedge d_D^c h  =  d\Theta.  \end{array} 
\end{equation}

\subsection{The reduced K\"ahler-Einstein equation}\label{derivation_of_system}
\subsubsection{From $X$ to $(\chi,h)$} We now assume that $n = 2$ and $D$ is a Riemann surface of genus at least two. Let $\omega_D$ be a K\"ahler-Einstein metric on $D$ normalized so that $\text{Ric} \, \omega_D = -\omega_D$. For the remainder of this paper, we assume that all K\"ahler-Einsten metrics are negative and normalized this way.  Because $D$ has complex dimension one, we can write $\tilde{\omega} = \chi \omega_D$ for some function $\chi$ on $D \times I$. Our goal is to reduce the equation $\text{Ric }\omega = -\omega$ to a simpler collection of equations on $\chi$ and $h$.

Let $\Omega$ be a local $S^1$-invariant holomorphic volume form on $X$ and let $\kappa = i(h dz + i\Theta + \kappa')$ be the $(1,0)$ form dual to $\xi^{1,0}$, where $\kappa'$ does not have a $dz$ or $dt$ component. We can write
\[ \Omega = \kappa \wedge \tilde{\Omega}\]
for $\tilde{\Omega} = i_{\xi^{1,0}}\Omega$ also $S^1$ invariant and compute that
\[ \Omega \wedge \bar{\Omega} = -2 i h \, dz \wedge dt \wedge \tilde{\Omega} \wedge \bar{\tilde{\Omega}}.\]
There are functions $\chi$ and $\sigma$ such that
\[ \omega_D = \sigma \, i \tilde{\Omega} \wedge \bar{\tilde{\Omega}} \ \ \ \ \ \ \ \tilde{\omega} = \chi \, \omega_D,\]
so
\begin{equation}\label{hexpression} \frac{\omega^2}{\Omega \wedge \bar{\Omega}} =  \frac{\tilde{\omega}}{ih \tilde{\Omega} \wedge \tilde{\bar{\Omega}}} = \frac{\chi \sigma}{h}, 
\end{equation}
and so
\[ \text{Ric}\, \omega = -i \partial \bar{\partial} \log \text{det}\frac{\omega^2}{\Omega \wedge \bar{\Omega}} = -i\partial \bar{\partial} (\log(\chi) + \log(\sigma) - \log(h)).\]
Thus the assumption that $\omega$ is K\"ahler-Einstein gives that
\begin{equation}\label{omegaExpansion}
-\omega =  -\frac{1}{2}dd^c \log \chi - i \partial_D \bar{\partial}_D \log \sigma +  \frac{1}{2}d d^c \log h.
\end{equation} 
For clarity we have replaced $i\partial \bar{\partial}$ with $\frac{1}{2} dd^c$ for functions that vary in $z$ and $i\partial_D \bar{\partial}_D$ for functions that do not.

We expand out this relation and separate into components to derive four equations relating $\chi$ and $h$. On the one hand, 
\[ - \omega = -(dz \wedge (-dt + \theta) + \tilde{\omega}) .\]
On the other hand, for any $t$-invariant function $F$ on $X$ we have
\begin{align*}
dd^cF &= dJ(F_z dz + d_D F) \\
&=  d(F_z h^{-1}(-dt + \theta) + d_D^c F) \\
&=  (F_z h^{-1})_z dz \wedge (-dt + \theta)  + d_D(F_z h^{-1}) \wedge (-dt + \theta) \\ &\ \ \ \ \ \ \ \ + (F_z h^{-1}) (dz \wedge \partial_z\theta_z + d_D \theta)  + dz\wedge d_D^c (F_z) + d_D d_D^c F \\ 
&=  (F_z h^{-1})_z dz \wedge (-dt + \theta)  + d_D(F_z h^{-1}) \wedge (-dt + \theta) \\ &\ \ \ \ \ \ \ \ + (F_z h^{-1}) (dz \wedge -d_D^c h + \partial_z \tilde{\omega})  + dz\wedge d_D^c (F_z) + d_D d_D^c F.
\end{align*}
Let $F = \log h - \log \chi$. Then Equation \ref{omegaExpansion} becomes
\begin{align*}-2(dz \wedge (-dt + \theta) + \tilde{\omega}) &= ((\log h - \log \chi)_z h^{-1})_z dz \wedge (-dt + \theta)   \\ &\ \ \ \ \ \ \ \ + d_D((\log h - \log \chi)_z h^{-1})\wedge (-dt + \theta) \\ &\ \ \ \ \ \ \ \  + (\log h - \log \chi)_z h^{-1} (dz \wedge -d_D^c h + \partial_z \tilde{\omega}) \\ &\ \ \ \ \ \ \ \ + dz\wedge d_D^c ((\log h - \log \chi)_z) \\ &\ \ \ \ \ \ \ \ + d_D d_D^c (\log h-\log \chi)  - 2 i \partial_D \bar{\partial}_D \log \sigma. 
\end{align*}
The covector fields $dz, \ dt, \ dy$ and $d\bar{y}$ form a local basis for forms on $X$. Collecting components involving only $dz \wedge dt$ yields
\begin{equation}\label{dzdt} 
((\log h-\log \chi)_z h^{-1})_z  = -2.
\end{equation}
Since $\theta$ does not have a $dt$ component, collecting components involving only $dt$ wedged with either $dy$ or $d\bar{y}$ yields
\begin{equation}\label{dtdw} 
d_D((\log h-\log \chi)_z h^{-1}) = 0.
\end{equation}
Collecting terms in only $dy$ and $d\bar{y}$ yields 
\begin{equation}\label{dwdwbar}
(\log h-\log \chi)_z h^{-1} \partial_z \tilde{\omega} + d_D d_D^c (\log h-\log \chi) - 2i \partial_D \bar{\partial}_D \sigma   =  - 2\tilde{\omega}
\end{equation}
Finally, cancelling all of these terms and ``dividing'' by $dz$, we have
\begin{equation}\label{leftover}
-(\log h-\log \chi)_z h^{-1} \,  d_D^c h + d_D^c (\log h-\log \chi)_z = 0,
\end{equation}
but this is trivial by Equation \ref{dtdw}.

\subsubsection{From $(\chi,h)$ to $X$}\label{backwards} The key point is that we can also work backwards from these equations. Fix $(D, \omega_D)$ a complex curve with $\text{Ric } \omega_D = -\omega_D$. For the remainder of this paper we will assume that $k_- = 0$ and $k_+ = -1$. We will discuss this simplification further in Section \ref{hopf_identification}. 

Fix a point $p_D \in D$ and let $\delta_p$ be a delta function at $(p_D,0) \in D \times [-1,1/2]$, i.e. 
\begin{equation}\label{delta_function_defn}
\int_{D \times [-1,1/2]} f\, \delta_p \, dz \wedge \omega_D = f(p_D,0).
\end{equation}
Since we are free to add a constant to the moment map $z$, it is sufficient to find a pair $(\chi, h)$ solving the following equations: 
\begin{equation}\label{maineqn1}
(\log h - \log \chi)_z = -2hz
\end{equation} 
\begin{equation}\label{maineqn2}
\partial_z^2 \chi + \Delta_D  h = 2 \pi \delta_p
\end{equation} 
\begin{equation}\label{maineqn3}
\tilde{\omega}-(\omega_D + z \partial_z \tilde{\omega}) =  -\frac{1}{2} d_D d_D^c (\log h - \log \chi).
\end{equation} 
Equation \ref{maineqn1} comes from integrating Equations \ref{dzdt} and \ref{dtdw}. Equation \ref{maineqn2} is part of Equation \ref{kahlerCondition}, except that we have added a delta function at $z = 0$. We will explain this choice in Section \ref{m_construction}. Equation \ref{maineqn3} is the result of substituting Equation \ref{maineqn1} into Equation \ref{dwdwbar}. 

The second line of Equation \ref{kahlerCondition} gives an additional constraint. Define \[\Gamma = \partial_z \tilde{\omega} - dz \wedge d_D^c h.\] If $\frac{1}{2\pi} \Gamma$ is an integral $(1,1)$-form and $z$ is defined on $[-1,1/2]$, there exists an $S^1$ bundle $\pi: \mathcal{M} \rightarrow D \times [-1,1/2]$ with connection form $-i\Theta$ and curvature $-i \Gamma$. Choices of such $\Theta$ modulo gauge equivalence are parameterized by $\text{Hom}(H^1(\mathcal{M}),S^1)$ and yield distinct complex structures on $\mathcal{M}$. 

Given such a $\chi$, $h$, and $\Gamma$, and fixing a choice of $\Theta$, the metric on $\mathcal{M}$ is given by $\omega = \tilde{\omega} + dz \wedge \Theta$. The computation above shows that $\mathcal{L}_{\xi} \omega = 0$. Equations \ref{maineqn2} and the second line in Equation \ref{kahlerCondition} give that the complex structure defined by the complex structure on $D$ and the condition that $J dz = h \Theta$ is integrable and also that $\tilde{\omega}$ is K\"ahler. Finally, Equations \ref{maineqn1} and \ref{maineqn3} ensure that $\tilde{\omega}$ is K\"ahler-Einstein.

Note that by our addition of a delta function to Equation \ref{maineqn2}, $\chi$ and $h$ will be singular at a fixed $p = (p_D,0) \in D \times [-1,1/2]$. Since all of our computations have been pointwise, our construction goes through without change on $D \times [-1,1/2] \setminus \lbrace p \rbrace$. Though \textit{a priori} the $S^1$ bundle $\mathcal{M}$ is defined only over $D \times [-1,1/2] \setminus \lbrace p \rbrace$, we will see in Section \ref{compactification} that it can be completed over the singular point. The singularity in Equation \ref{maineqn2} causes a change in the degree of the $S^1$ bundle restricted to $D$ as we pass through $z = 0$. This is necessary to match the Calabi model spaces at either end. 

\subsection{Deriving a linearized equation for $\delta h$.}
We may ask whether there is a solution such that $h$ and $\chi$ are constant over $D$ for each $z$. In this case Equation \ref{maineqn2} tells us that away from the singularity, \[\chi = az + b \] for constants $a$ and $b$. In fact, Equation \ref{maineqn3} tells us that $b = 1$. Now if we write $u = \frac{h}{az+1}$, then we can rewrite Equation \ref{maineqn1} as
\[ (\log u)_z = -2u(az+1)z.\]
We find that
\[ h = \frac{a z+1}{\frac{2}{3}az^3 + z^2 +c}\]
for some constant $c$.  

Motivated by this observation, we guess that a nontrivial approximate solution takes the form
\begin{equation}\label{ansatz} 
h = h_0 + \delta h, \ \ \ \ \ \ \ \ \ \ \ \ \chi = 1 + \delta \chi 
\end{equation}
for $h_0 = \frac{1}{z^2+T^{-2}}$, functions $\delta h$ and $\delta \chi$ that are small in some sense, and $T$ a large constant. Then Equation \ref{maineqn1} becomes
\[ \partial_z \bigg( \log(h_0) + \log\left(1+\frac{\delta h}{h_0}\right) - \log(1+\delta\chi)\bigg) = -2(h_0 + \delta h)z.\] 
Since $h_0$ solves Equation \ref{maineqn1} with $\chi \equiv 1$, it follows that
\[ \partial_z \bigg( \log\left(1+\frac{\delta h}{h_0}\right) - \log(1+\delta\chi)\bigg) = -2z \, \delta h.\] 
Expanding the logarithmic terms into a power series and keeping only terms linear in $\delta h$ and $\delta \chi$, we derive that
\begin{equation}\label{linearized_before_deriv} \partial_z\bigg( \frac{\delta h}{h_0} \bigg)-\delta\chi_z = -2z\,\delta h.
\end{equation}
Taking another derivative yields that
\[ (z^2+T^{-2})\delta h_{zz} + 4z \, \delta h_z + 2\delta h-\delta\chi_{zz} = -2\delta h-2z\,\delta h_z.\]
By Equation \ref{maineqn2}, this becomes
\begin{equation}\label{deltah_eqn} (z^2+T^{-2}) \delta h_{zz} + 6z \, \delta h_z + (4 + \Delta_D) \delta h = 2 \pi \delta_p.
\end{equation}
Equation \ref{deltah_eqn} is the linearization of the K\"ahler-Einstein equation under $S^1$ symmetry.

\section{Constructing an Approximate Solution}\label{solution}
In this section we solve Equation \ref{deltah_eqn}. In the notation introduced in Section \ref{kahler_reduction}, this gives us a candidate for $h$, and we can define $\chi$ using Equation \ref{maineqn2}. Since the resulting pair $(\chi, h)$ gives an exact solution to Equation \ref{maineqn2}, the resulting space is K\"ahler, but since we have linearized Equation \ref{maineqn1}, it is not Einstein. In Section \ref{perturbation} we will check that the discrepancy between the metric and a negative multiple of its Ricci curvature is not too large and then correct to an honest K\"ahler-Einstein space. 
\subsection{Solving the linearized equation for $\delta h$}
\subsubsection{Separation of variables}
We assume that the solution can be expanded in the eigenfunctions of $\Delta_D$ and write
\begin{equation}\label{deltah_expansion}
\delta h = \sum_{\lambda \geq 0} f_{\lambda} \psi_{\lambda}
\end{equation}
where $\Delta_D \psi_{\lambda} = -\lambda^2 \psi_{\lambda}$, $\psi_{\lambda}(p_D) \geq 0$, and $\int_D \vert \psi_{\lambda} \vert^2 \, \omega_D = 1$. We then have the formal expansion
\[  \delta_p = \sum_{\lambda} \psi_{\lambda}(p_D) \psi_{\lambda} \delta_0, \]
where $\delta_0 = \delta_0(z)$ is a delta function at $0$ on $[-1,1/2]$. Matching terms, Equation \ref{deltah_eqn} then gives an ordinary differential equation 
\begin{equation}\label{mode_ode}
(z^2+T^{-2}) f_{\lambda}''(z) + 6z f_{\lambda}'(z) + (4 -\lambda^2) f_{\lambda}(z)= 2 \pi \psi_{\lambda}(p_D)  \delta_0(z)
\end{equation}
for each eigenvalue $\lambda$. 

Now if $\Vert f_{\lambda} \Vert_{L^2(\mathbb{R})} < \infty$ for some $\lambda > 0$ then for each test function $\chi \in C^{\infty}_0(D \times [-1,1/2])$, 
\[ \langle f_{\lambda} \psi_{\lambda},\chi \rangle = \int_{-1}^{1/2} f_{\lambda}(z)  \bigg( \int_D \psi_{\lambda}(\cdot) \, \chi(\cdot,z)   \, \omega_D \bigg) dz.\] 
But by elliptic regularity,
\begin{equation}\label{psi_reg}
\Vert \psi_{\lambda} \Vert_{C^0} = \mathcal{O}(\sqrt{\lambda}).
\end{equation}
Therefore
\begin{align*} \bigg\vert \int_D \psi_{\lambda}(\cdot) \, \chi(\cdot,z)  \, \omega_D \bigg\vert &= \bigg\vert \frac{1}{\lambda^{2\ell}} \int_D \Delta_D^\ell \psi_{\lambda}(\cdot) \, \chi(\cdot,z)  \, \omega_D \bigg\vert \\ 
&= \bigg\vert \frac{1}{\lambda^{2\ell}} \int_D \psi_{\lambda}(\cdot) \, \Delta_D^\ell \chi(\cdot,z)  \, \omega_D \bigg\vert \\ 
&\leq \frac{C}{\lambda^{2\ell-\frac{1}{2}}} \Vert \chi \Vert_{C^{2\ell}(D \times I)},
\end{align*}
where $C$ does not depend on $\lambda$. Now say that $\Vert f_{\lambda} \Vert_{L^2(\mathbb{R})}$ grows slower than $\lambda^N$ for some $N$. Weyl's law guarantees that 
\begin{equation}\label{Weyl}
\# \lbrace \lambda \text{ an eigenvalue of $\Delta_D$}: \lambda \in [k-1,k) \rbrace \leq Ck 
\end{equation}
for some $C > 0$. Therefore for any $M > 0$, 
\begin{align*}
\left\langle \sum_{1 \leq \lambda \leq M} f_{\lambda} \psi_{\lambda},\chi \right\rangle &\leq  C
\sum_{1 \leq \lambda \leq M} \frac{ \Vert f_{\lambda} \Vert_{L^2(\mathbb{R})}}{\lambda^{2\ell-\frac{1}{2}}} \Vert \chi \Vert_{C^{2\ell}(D \times I)} \\
&\leq C  \sum_{k=2}^{\infty} k^{N+\frac{3}{2}-2\ell} \Vert \chi \Vert_{C^{2\ell}(D \times I)}.
\end{align*}
Since there are only finitely many independent eigenvectors with $\lambda < 1$, choosing $\ell$ large enough proves that $\sum f_{\lambda} \psi_{\lambda}$ gives a well-defined distribution solving Equation \ref{deltah_eqn}. It remains to show that each $f_{\lambda}$ is defined on $[-1,1/2]$ and that $\Vert f_{\lambda} \Vert_{L^2(-1,1/2)}$ has polynomial growth in $\lambda$. 

\subsubsection{ODE solution for each eigenvalue.}\label{ode_each_eval} To do this, we solve Equation \ref{mode_ode} explicitly. The change of variables $x(z) = \frac{1}{2}(1+iTz)$ converts Equation \ref{mode_ode} to the hypergeometric equation
\begin{equation}\label{hypergeo_form}
x(x-1)f''(x) + (6x-3) f'(x) + (4-\lambda^2) f(x) =  T \pi \, \psi_{\lambda}(p_D)  \delta_0(x).
\end{equation}
Note that 
\begin{equation}\label{conj}
1-x(z) = \overline{x(z)}, \ \ \ \ z \in \mathbb{R}.
\end{equation}

For $n \in \mathbb{R}$, define \[ (n)_k = \begin{cases} \prod_{j=0}^{k-1} (n+j) & k \in \mathbb{Z}_{>0} \\   1 & k = 0 \end{cases}.\] Let $F$ be the hypergeometric series
\[ F(\alpha,\beta,\gamma;x) = \sum_{k=0}^\infty \frac{(\alpha)_k (\beta)_k}{k!(\gamma)_k} x^k.\]
For any $\alpha, \beta, \gamma \in \mathbb{R}$ with $\gamma \not \in \mathbb{Z}_{< 0}$, $F$ is absolutely convergent on the open unit disk in $\mathbb{C}$.  Now let
\[ \alpha(\lambda) =  \frac{5+\sqrt{9+4\lambda^2}}{2} \ \ \ \ \ \ \beta(\lambda) =  \frac{5-\sqrt{9+4\lambda^2}}{2} \ \ \ \ \ \ \gamma = 3. \]
Using the recursion in the coefficents of $F$, it is checked that for all $\lambda$, the functions 
\[ v_1^{\lambda} = F(\alpha(\lambda), \beta(\lambda), \gamma(\lambda); x(z)) \ \ \text{ and } \ \ v_2^{\lambda} = F(\alpha(\lambda), \beta(\lambda), \gamma(\lambda);\overline{x(z)})\]
solve the homogeneous version of Equation \ref{mode_ode}. Now it is easily observed from the definition that
\[ \frac{d}{dx} F(\alpha(\lambda),\beta(\lambda),\gamma;x) = \frac{\alpha(\lambda) \beta(\lambda)}{\gamma} F(\alpha(\lambda)+1,\beta(\lambda)+1,\gamma+1;x),\] so the Wronskian $W[v_1^{\lambda},v_2^{\lambda}](0) = {v_1^{\lambda}}(0){v_2^{\lambda}}'(0) - {v_1^{\lambda}}' (0){v_2^{\lambda}}(0)$ is evaluated as
\[ -\frac{iT \alpha(\lambda)\beta(\lambda)}{\gamma} F(\alpha+1,\beta+1,\gamma+1,\frac{1}{2})F(\alpha,\beta,\gamma,\frac{1}{2}). \]
An identity due to Gauss states that
\begin{equation}\label{gaussIdentity} 
F\left(\alpha,\beta,\frac{1+\alpha+\beta}{2};\frac{1}{2}\right) = \frac{\Gamma(\frac{1}{2})\Gamma(\frac{1+\alpha+\beta}{2})}{\Gamma(\frac{1+\alpha}{2})\Gamma(\frac{1+\beta}{2})}
\end{equation}
where $\Gamma$ is the gamma function (\cite{Luke} 3.13.2). Thus for our choices of $\alpha$, $\beta$, and $\gamma$,
\[ W[v_1^{\lambda},v_2^{\lambda}](0) = -\frac{iT\alpha(\lambda) \beta(\lambda)}{\gamma} \frac{\Gamma(\frac{1}{2})^2\Gamma(\frac{1+\alpha(\lambda)+\beta(\lambda)}{2}) \Gamma(\frac{3+\alpha(\lambda)+\beta(\lambda)}{2})}{\Gamma(\frac{1+\alpha(\lambda)}{2}) \Gamma(\frac{2+\alpha(\lambda)}{2}) \Gamma(\frac{1+\beta(\lambda)}{2}) \Gamma(\frac{2+\beta(\lambda)}{2})}.\] 
Now $\Gamma$ has no roots but has a simple pole with residue $\frac{(-1)^k}{k!}$ at $-k$ for $k = 0,1,2...$. The sum $\alpha(\lambda) + \beta(\lambda) = 5$ for all $\lambda$, so the numerator of the second term in the above product is constant in $\lambda$. For all $\lambda$, $\alpha(\lambda)$ is positive, so the $\Gamma$ terms involving $\alpha(\lambda)$ are finite positive numbers. However, we may have that $\frac{1+\beta(\lambda)}{2}$ or $\frac{2+\beta(\lambda)}{2}$ is equal to a negative integer for infinitely many $\lambda$. 

Let
\[ \Sigma = \lbrace \lambda > 0: \ \lambda \ \text{is an eigenvalue of }\Delta_D \text{ and } \sqrt{9+4\lambda^2} \in \mathbb{Z}\rbrace.\] 
Then $\Sigma$ is discrete but possibly infinite. For $\lambda \not \in \Sigma$, $v_1^{\lambda}$ and $v_2^{\lambda}$ give independent solutions to the homogeneous version of Equation \ref{mode_ode}. In this case, we can solve the non-homogeneous equation by variation of parameters. We write
\[ v_{\lambda} = c_1 v_1^{\lambda} + c_2 v_2^{\lambda},\]
where 
\begin{dmath*} c_1(z) = -\int_{-\infty}^z \frac{2 \pi \psi_{\lambda}(p_D) \delta_0}{s^2+T^{-2}}\frac{v_2^{\lambda}(s)}{W[v_1^{\lambda},v_2^{\lambda}](s)} \, ds 
= -2  \pi i \,  \psi_{\lambda}(p_D) \sigma_0(z) \frac{T\gamma}{\alpha(\lambda) \beta(\lambda)}\frac{\Gamma(\frac{2+\alpha(\lambda)}{2})\Gamma(\frac{2+\beta(\lambda)}{2})}{\Gamma(\frac{1}{2})\Gamma(\frac{3+\alpha(\lambda)+\beta(\lambda)}{2})}
\end{dmath*}
and
\begin{dmath*} c_2(z) = \int_{-\infty}^z \frac{2 \pi  \psi_{\lambda}(p_D) \delta_0}{s^2+T^{-2}}\frac{v_1^{\lambda}(s)}{W[v_1^{\lambda},v_2^{\lambda}](s)} \, ds 
= 2 \pi i \,\psi_{\lambda}(p_D)  \sigma_0(z)  \frac{T\gamma}{\alpha(\lambda) \beta(\lambda)}\frac{\Gamma(\frac{2+\alpha(\lambda)}{2})\Gamma(\frac{2+\beta(\lambda)}{2})}{\Gamma(\frac{1}{2})\Gamma(\frac{3+\alpha(\lambda)+\beta(\lambda)}{2})}
= -c_1(z).
\end{dmath*}
Since $v_1(z) = \overline{v_2(z)}$, this gives that
\[ v_{\lambda} = 2i \, c_1(z) \, \text{Im} \, v_1^{\lambda}(z) =   \, 12 \pi \sigma_0(z) \frac{ T  \psi_{\lambda}(p_D)}{\alpha(\lambda) \beta(\lambda)}\frac{\Gamma(\frac{2+\alpha(\lambda)}{2})\Gamma(\frac{2+\beta(\lambda)}{2})}{\Gamma(\frac{1}{2})\Gamma(\frac{3+\alpha(\lambda)+\beta(\lambda)}{2})} \text{Im} \, v_1^{\lambda}(z),\]
where 
\[ \sigma_0(z) = \begin{cases}
0 & \text{if $z < 0$} \\
1 & \text{if $z \geq 0$}
\end{cases}.\]  
It is verified that $v_{\lambda}$ is a solution to Equation \ref{mode_ode} in the distributional sense.  

\subsubsection{Finding a decaying solution for $\lambda \not \in \Sigma$} 
We would like to define $f_{\lambda} = v_{\lambda}$, but for Equation \ref{deltah_expansion} to converge we must modify $v_{\lambda}$ by a linear combination of homogeneous solutions to Equation \ref{mode_ode} so that $f_{\lambda}(z)$ grows at worst like a fixed polynomial in $\lambda$ for all $z$. In this section we write $\alpha$ for $\alpha(\lambda)$ and $\beta$ for $\beta(\lambda)$ for brevity.

Fix a branch of the complex logarithm cut along the negative real axis. We have that for $\lambda \not \in \Sigma$, $F(\alpha,\beta,\gamma;x)$ has the analytic extension 
\begin{align*} F(\alpha,\beta,\gamma;x) =   \frac{\Gamma(\beta-\alpha) \Gamma(\gamma)}{\Gamma(\beta) \Gamma(\gamma-\alpha)} &(-x)^{-\alpha} F(\alpha,\alpha+1-\gamma, \alpha+1-\beta;x^{-1})  \\ &+   \frac{\Gamma(\alpha-\beta) \Gamma(\gamma)}{\Gamma(\beta) \Gamma(\gamma-\beta)} (-x)^{-\beta} F(\beta,\beta+1-\gamma, \beta+1-\alpha;x^{-1}),
\end{align*} 
which is defined for all $x$ such that $-x$ is in the domain of $\log$ (see \cite{Luke} 3.6(28)-(30)).  For clarity, define \[ f_1(x) =  \frac{\Gamma(\beta-\alpha) \Gamma(\gamma)}{\Gamma(\beta) \Gamma(\gamma-\alpha)} \vert x \vert^{-\alpha}e^{-\alpha i (\text{arg}(x) - \pi)}F(\alpha,\alpha+1-\gamma, \alpha+1-\beta;x^{-1})\] and \[ f_2(x) = \frac{\Gamma(\alpha-\beta) \Gamma(\gamma)}{\Gamma(\beta) \Gamma(\gamma-\beta)} \vert x \vert^{-\beta}e^{-\beta i (\text{arg}(x) - \pi)} F(\beta,\beta+1-\gamma, \beta+1-\alpha;x^{-1})\]
on $\mathbb{C} \setminus \mathbb{R}$. Then if $z > 0$,
\begin{equation}\label{upperhalf} F(\alpha,\beta,\gamma;x) = f_1(x) + f_2(x),
\end{equation}
while if $z < 0$,  
\begin{equation}\label{lowerhalf} F(\alpha,\beta,\gamma;x) = e^{-2\pi i\alpha}f_1(x) + e^{-2\pi i\beta}f_2(x).
\end{equation} 
Now because $f_1$ and $f_2$ are linearly independent solutions to the homogeneous version of Equation \ref{mode_ode} (\cite{Luke} 3.7) and $F(\alpha,\beta,\gamma;1-x)$ is defined on the upper half-plane, there exist $\mu_1$ and $\mu_2$ such that 
\begin{equation}\label{combo1} F(\alpha,\beta,\gamma;1-x(z)) = \mu_1 f_1(x(z)) + \mu_2 f_2(x(z))
\end{equation}
for $z > 0$. But by Equations \ref{conj} and \ref{lowerhalf} we have that if $z > 0$, 
\begin{equation}\label{combo2}
 F(\alpha,\beta,\gamma;1-x(z)) = e^{-2\pi i\alpha}f_1(\bar{x}(z)) + e^{-2\pi i\beta}f_2(\bar{x}(z)).
\end{equation} 
Let $z \rightarrow +\infty$ and divide Equation \ref{combo1} by Equation \ref{combo2}. Note that as $\vert z \vert \rightarrow \infty$, $F(\beta,\beta+1-\gamma,\beta+1-\alpha;x^{-1}) \rightarrow 1$ and likewise for the conjugate. Since $\beta < \alpha$, $\vert x \vert^{-\beta}$ dominates $\vert x \vert^{-\alpha}$, while the phase of $x$ goes to $\pi/2$, so we have that
\[ e^{2\pi i \beta} \mu_2 = \lim_{z \rightarrow +\infty} \frac{f_2(\bar{x}(z))}{f_2({x}(z))} = e^{\pi i\beta}.\]
But since $F(\alpha,\beta,\gamma;x)$ is a real power series, $F(\alpha,\beta,\gamma;1-x(z)) = \overline{F(\alpha,\beta,\gamma;x(z))}$. Thus Equations \ref{upperhalf} and \ref{combo1} give that 
\[ \lim \limits_{z \rightarrow \infty} \frac{\text{Im}(F( \alpha,\beta,\gamma;x(z)))}{\text{Re}(F(\alpha,\beta,\gamma;x(z)))} = \tan\left(\frac{\pi}{2}\beta\right).\]
Taking the complex conjugate, we must have that 
\[ \lim \limits_{z \rightarrow -\infty} \frac{\text{Im}(F( \alpha,\beta,\gamma;x(z)))}{\text{Re}(F(\alpha,\beta,\gamma;x(z)))} = -\tan\left(\frac{\pi}{2}\beta\right).\]
Thus if we let $\rho = -\tan(\frac{\pi}{2}\beta )$, we have that
\[ \left(\sigma_0(z)-\frac{1}{2}\right)\text{Im}(v_1^{\lambda})(z) + \frac{\rho}{2}  \text{Re}(v_1^{\lambda})(z) = \mathcal{O}((Tz)^{-\alpha})\]

Now let 
\[ C_1 =  \frac{12 \pi \psi_{\lambda}(p_D) T}{ \alpha \beta}\frac{\Gamma(\frac{2+\alpha}{2})\Gamma(\frac{2+\beta}{2})}{\Gamma(\frac{1}{2})\Gamma(\frac{3+\alpha+\beta}{2})},\]
so that 
\[ v_{\lambda}(z) = C_1 \sigma_0(z) \text{Im}(v_1^{\lambda}(z)).\]  
Define
\begin{equation}\label{flambda_expression} f_{\lambda}(z)=\frac{C_1}{2}\bigg( -\text{Im}(v_1^{\lambda})(z)+\rho\text{Re}(v_1^{\lambda})(z) \bigg) + v_{\lambda}(z). 
\end{equation} 
By our discussion above, 
\begin{equation}\label{soln_decay}
f_{\lambda}(z) =\mathcal{O}((Tz)^{-\alpha}).
\end{equation}
Indeed our choice of $f_{\lambda}$ is the unique decaying solution to the inhomogeneous equation \ref{mode_ode}.

\subsubsection{Extension to $\Sigma$}  
By Equation \ref{conj}, we can replace $\text{Im}(F)$ and $\text{Re}(F)$ in Equation \ref{flambda_expression} with appropriate linear combinations of $F(\alpha, \beta, \gamma, x)$ and $F(\alpha, \beta, \gamma, 1-x)$. It can be checked that both of these functions are locally holomorphic in $x$, $\alpha$ and $\beta$. Checking all other terms in Equation \ref{flambda_expression}, we find that if we ignore the step function term, the function
\[ \frac{f_{\lambda}(x)}{\psi_{\lambda}(p_D)}, \]
extends to a meromorphic function of $\alpha$ (recall that $\beta = 5-\alpha$) and of $x$ away from $x = 1$. Call this function $g(\alpha,x)$. Now in Section \ref{l2_estimate} we will show that away from $\lambda = 0$, $f_{\lambda}(z)$ is uniformly bounded by a polynomial in $\lambda$ for $\lambda \not \in \Sigma$ and $z \in \mathbb{R}$. Therefore for $\lambda^{\ast} \in \Sigma$, $g$ is holomorphic in a neighborhood of \[\lbrace \alpha(\lambda^{\ast}) \rbrace \times \lbrace x(z): z \in [-1,0] \rbrace \subset \mathbb{C}^2,\] and so $g(\alpha(\lambda),x(z))$ converges smoothly to a function of $z \in (-1,0)$ as $\lambda \rightarrow \lambda^{\ast}$. Note that convergence holds up to the boundary, i.e. $\partial_z^k g(\alpha(\lambda),x(0)) \rightarrow \partial_z^k g^{k}(\alpha(\lambda^\ast),x(0))$ for all integers $k \geq 0$. 

By a similar argument, $g(\alpha,x)$  converges smoothly to a function for $z \in [0,1/2]$. We define 
\[ f_{\lambda^\ast}(z) = \psi_{\lambda^{\ast}}(p_D) g(\alpha(\lambda^{\ast}),z).\]
Smooth convergence up to the boundary on $[-1,0]$ and $[0,1/2]$ ensures that $f_{\lambda^{\ast}}$ satisfies Equation \ref{mode_ode}.

\subsubsection{Solution for $\lambda =0$}  This convergence argument does not give the solution at $\lambda = 0$ because the uniform bound on $g$ does not hold as $\lambda \rightarrow 0$. In this case, however, the general solution to the homogeneous version of Equation \ref{mode_ode} is given by
\[ \frac{1}{(1+(Tz)^2)^2}\bigg( \mu_1 + \mu_2 \bigg(z+\frac{T^2z^3}{3} \bigg) \bigg)\]
for $\mu_1, \ \mu_2 \in \mathbb{R}$. 

Variation of parameters yields an inhomogeneous solution
\[ v_0 =  2\pi \sigma_0(z) T^2 \psi_{0}(p_D) \frac{z+\frac{T^2z^3}{3}}{(1+(Tz)^2)^2}  .\]
We add a homogeneous solution to define the even function 
\[ f_0(z) = (2\sigma_0(z)-1)\pi T^2\psi_{0}(p_D) \frac{z+\frac{T^2z^3}{3}}{(1+(Tz)^2)^2}.\]

\subsubsection{$L^2$ estimates of the ODE solutions}\label{l2_estimate}
To show the convergence of Equation \ref{deltah_expansion}, we first control the growth of the solutions to Equation \ref{mode_ode} at the origin. Then we use the maximum principle to show that the value at the origin bounds the value everywhere. 
\begin{proposition}\label{zero_bound}For all $\lambda > 2$ and $T > 0$, let $f_{\lambda}^T$ be the unique decaying solution to Equation \ref{mode_ode}. Then $f_{\lambda}^T(0) \leq 0$. Further, there exists $N > 0$ independent of $T$ and $\lambda$ such that 
\[ \vert f_{\lambda}^T(0) \vert = T \mathcal{O}(\lambda^N)\]
as $\lambda \rightarrow \infty$. 
\end{proposition}  
\begin{proof} 
Using Equation \ref{gaussIdentity} to evaluate $v_1^{\lambda}(0)$ and noticing that $v_1^{\lambda}(0)$ is a real power series with a real argument, we see that if $\lambda \not \in \Sigma$ then
\begin{dmath}\label{G}
f_{\lambda}^T(0) = -\frac{12 \pi T \psi_{\lambda}(p_D)}{2 \alpha\beta} \frac{\Gamma(\frac{2+\alpha}{2})\Gamma(\frac{2+\beta}{2})}{\Gamma(\frac{1}{2})\Gamma(4)} \frac{\Gamma(\frac{1}{2})\Gamma(3)} {\Gamma(\frac{1+\alpha}{2})\Gamma(\frac{1+\beta}{2})} \tan\left(\frac{\pi}{2} \beta\right)
= -C\frac{T\psi_{\lambda}(p_D)}{\alpha\beta} \frac{\Gamma(\frac{2+\alpha}{2})\Gamma(\frac{2+\beta}{2})}{\Gamma(\frac{1+\alpha}{2})\Gamma(\frac{1+\beta}{2})} \tan\left(\frac{\pi}{2} \beta\right)
\end{dmath}
for some $C > 0$ independent of $\lambda$ and $T$. First we note that by elliptic regularity,\[\psi_{\lambda}(p_D) = \mathcal{O}(\sqrt{\lambda})\] as $\lambda \rightarrow \infty$.
To understand the gamma function terms, we note that for $x \geq 3$, 
\begin{equation}\label{gamma_squeeze}
\lfloor x-1 \rfloor ! \leq \Gamma(x)  \leq \lceil x-1 \rceil !.
\end{equation}
In addition, we have the following identities (see \cite{Luke}  2.2).
\begin{lemma}  \label{gamma_identities}
For $x \in \mathbb{R}$ and $N \in \mathbb{Z}_{> 0}$, 
\begin{equation}\label{gamma_intshift}
\Gamma(x+N) = (x)_N \Gamma(x),
\end{equation}
\begin{equation}\label{gamma_one_reflect}
\Gamma(x)\Gamma(1-x) = \pi \csc \pi x,
\end{equation}
and 
\begin{equation}\label{gamma_half_reflect}
\Gamma(\frac{1}{2}+x)\Gamma(\frac{1}{2}-x) = \pi \sec \pi x.
\end{equation}
\end{lemma}

Equation \ref{gamma_one_reflect} and \ref{gamma_half_reflect} give that
\begin{dmath*} \frac{\Gamma(\frac{2+\alpha}{2})\Gamma(\frac{2+\beta}{2})}{\Gamma(\frac{1+\alpha}{2})\Gamma(\frac{1+\beta}{2})} = 
-\frac{\Gamma(\frac{2+\alpha}{2})\Gamma(\frac{1-\beta}{2})}{\Gamma(\frac{1+\alpha}{2})\Gamma(\frac{-\beta}{2})} \cot\left(\frac{\pi \beta}{2}\right).
\end{dmath*}  
Cancelling the tangent terms and noting that $\beta < 0$ as long as $\lambda > 2$, this proves that $f_{\lambda}^T(0) \leq 0$. Further, since $a \rightarrow \infty$ as $\lambda \rightarrow \infty$ and $\beta = 5-\alpha$, Equation \ref{gamma_squeeze} gives that for large enough $\lambda$, 
\[  \frac{\Gamma(\frac{2+\alpha}{2})\Gamma(\frac{1-\beta}{2})}{\Gamma(\frac{1+\alpha}{2})\Gamma(\frac{-\beta}{2})} \leq \frac{\lceil \frac{\alpha}{2} \rceil! \lceil \frac{\alpha-6}{2} \rceil!}{ \lfloor \frac{\alpha-1}{2} \rfloor ! \lfloor \frac{\alpha-7}{2} \rfloor!},\]
which by Equation \ref{gamma_intshift} is bounded by a polynomial in $\alpha$ as $\alpha \rightarrow \infty$. But $\alpha  = \mathcal{O}(\lambda)$ as $\lambda \rightarrow \infty$.   
\end{proof}

\begin{proposition}\label{monotonicity} Take $f_{\lambda}^T$ as in Proposition \ref{zero_bound}. Then $f_{\lambda}^T(z)$ is  nonincreasing on $(-\infty,0]$ and nondecreasing on $[0,\infty)$.
\end{proposition}
\begin{proof} 
We first note that $f_{\lambda}^T(z) \leq 0$ for all $z$ and $\lambda > 2$. For away from $z = 0$, $f_{\lambda}^T$ is a smooth solution to the homogeneous version of Equation \ref{mode_ode}. Therefore by Proposition \ref{zero_bound} and Equation \ref{soln_decay}, if $f_{\lambda}^T(z) > 0$ for any $z < 0$ then $f_{\lambda}^T$ has a positive local maximum at some $z^{\ast} \in (-\infty,0)$. But then $(f_{\lambda}^T)''(z^{\ast}) \leq 0$, $(f_{\lambda}^T)'(z^{\ast}) = 0$, and $f_{\lambda}^T(z^{\ast}) > 0$, contradicting Equation \ref{mode_ode}. An identical argument shows that $f_{\lambda}^T$ is nonpositive on $(0,\infty)$. 

Now we show that for $\lambda> 2$, $f_{\lambda}^T(z)$ is nonincreasing on $(-\infty,0]$. An argument as in the previous paragraph shows that $f_{\lambda}^T(z)$ cannot have a negative local minimum on $(-\infty,0)$. But if $f_{\lambda}(z_1) < f_{\lambda}(z_2)$ for $z_1 < z_2 < 0$, there must be a negative local minimum on $(-\infty,z_2)$. 

A similar argument gives that $f_{\lambda}^T$ is nondecreasing from $0$ to $\infty$. 
\end{proof} 
Propositions \ref{zero_bound} and \ref{monotonicity} give the desired $L^2$ growth control.  
\begin{proposition} Fix $T>0$. Then \[ \Vert f_{\lambda}^T(z) \Vert_{L^2(-1,1)} = \mathcal{O}(\lambda^N)\] for some $N > 0$.
\end{proposition} 
\begin{remark}\label{soln_extension}
The solutions $f_{\lambda}^T$ are defined on all of $\mathbb{R}$, and the convergence arguments of this section apply to any compact interval. Therefore it is justified to speak of $\delta h$ as a function on $\mathbb{R}$.
\end{remark}

\subsubsection{Decay of $\delta h$}
Away from the singular point, the behavior of $\delta h$ is controlled by the zero mode which is constant over $D$.
\begin{proposition}\label{deltah_decay_prop}
Let $\delta h^T = \sum_{\lambda \geq 0} f_{\lambda}^T \psi_{\lambda}$. Then there exists $\rho > 0$ such that
\begin{equation}\label{deltah_decay}
\delta h^T= f_0^T \psi_0 + T \,\mathcal{O}((Tz)^{-4})
\end{equation}
and
\begin{equation}\label{deltah_decay_laplac}
\Delta_D \delta h^T = T \,\mathcal{O}((Tz)^{-4})
\end{equation}
on $\vert z \vert \geq \rho T^{-1}$ as $T \rightarrow \infty$ . 
\end{proposition}
\begin{proof}
We see from Equation \ref{mode_ode} that $\frac{1}{T}\delta h^T(z) = \delta h^1(Tz)$. Therefore it is sufficient to show that 
\[ \delta h^1(z) = f_0^1 \psi_0 + \mathcal{O}(z^{-4})\]
for $\vert z \vert \geq \rho$. By Remark \ref{soln_extension} we  may treat $\delta h^1$ as a function on $\mathbb{R}$. If $\lambda > 0$ then $\alpha(\lambda) > 4$, so by Equation \ref{soln_decay} each  $\psi_{\lambda} f_{\lambda}(z)$ is dominated by $z^{-4} C_{\lambda}$ for some  $C_{\lambda} > 0$. It remains to show that the series $\sum_{\lambda} C_{\lambda}$ converges.

Now for $z \in (0,\infty)$ and $\epsilon \in (0,1)$ let \[g_{\lambda}(z) = z^{\epsilon\lambda+4} \max_D \vert \psi_{\lambda} \vert \, f_{\lambda}(z).\] Then Equation \ref{mode_ode} gives that
\begin{equation}\label{geqn} 
z^2(1+z^2)g_{\lambda}''(z) \ \ + \ \  ...  
\ \  + \ \ 
(\lambda^2(\epsilon^2 + (\epsilon^2-1)z^2) + \lambda(3\epsilon z^2 + 9 \epsilon) + 20) g_{\lambda}(z) \equiv 0,
\end{equation}
on $D$, where we have omitted the $g'_{\lambda}(z)$ term. The coefficient of $g_{\lambda}''(z)$ and the coefficient of $g_{\lambda}(z)$ are positive for all $\vert z \vert < 1$ for $\epsilon$ close enough to $1$. 

Assume that $\lambda \gg 0$. By Proposition \ref{zero_bound}, $g_{\lambda}(z) < 0$ for all $z > 0$. Equation \ref{soln_decay} gives that  $\lim \limits_{z \rightarrow \pm \infty} \vert g_{\lambda}(z) \vert = 0$, since \[\alpha(\lambda) > 4 + \epsilon\lambda, \ \ \lambda \gg 0.\] Also, since $f_{\lambda}$ is continuous, $\lim \limits_{z \rightarrow 0} g_{\lambda}(z) = 0$. But by Equation \ref{geqn} and the arguments from Proposition \ref{monotonicity}, $g_{\lambda}(z)$ cannot have a negative local maximum in $(0,\infty)$. Therefore there exists $\eta > 0$ such that $g_{\lambda}(z)$ is nondecreasing on $(\eta,\infty)$. 

Fix $\rho$ and $\rho_1$ such that $ \eta < \rho_1 < \rho$. By Proposition \ref{zero_bound}, Proposition \ref{monotonicity}, and Equation \ref{psi_reg}, \[\vert \max_D \psi_{\lambda}\, f_{\lambda}(\rho_1) \vert = \mathcal{O}(\lambda^k)\] for some $k < \infty$. Then 
\[
\vert \rho^4 \, \max_D \psi_{\lambda}\, f_{\lambda}(\rho) \vert \leq \bigg\vert \bigg(\frac{\rho_1}{\rho}\bigg)^{\epsilon\lambda} \rho_1^4  \max_D \psi_{\lambda}\, f_{\lambda}(\rho_1)   \bigg\vert
\leq C \bigg( \frac{\rho_1}{\rho}\bigg)^{\epsilon\lambda} \lambda^k,
\]
where $C$ depends on $\rho_1$ but not $\lambda$. Using Equation \ref{Weyl}, the right hand side is summable, so we can take $C_{\lambda} = C (\frac{\rho_1}{\rho})^{\epsilon\lambda} \lambda^k$ with $C$ as above. By nondecreasingness, this bound holds for all $z > \rho$. Meanwhile, because $f_{\lambda}$ is even-symmetric, the bound holds for $z < -\rho$ as well.


Since $\Delta_D \delta h = \sum_{\lambda > 0} -\lambda^2 \psi_{\lambda} f_{\lambda}$, the proof for $\Delta_D \delta h^T(z)$ follows with $C_{\lambda} = C (\frac{\rho_1}{\rho})^{\epsilon\lambda} \lambda^{k+2}$. 
\end{proof}

\subsection{Correction to $h$}\label{hcorrection} Given the solution $h = h_0 + \delta h$ to the linearization of Equation \ref{maineqn1} that we have just constructed, we can define $\chi$ by integrating Equation \ref{maineqn2} and choosing $a = k_- = 0$. Having done so, we see that $(\chi,h + q(z))$ also solves Equation \ref{maineqn2} for $q$ a smooth function of $z$. We make use of this freedom to construct another solution $\tilde{h}$ that has the desired behavior near the ends of the $z$-interval $[-1,1/2]$. 

Fix a constant $C_2 > 0$ as in Proposition \ref{deltah_decay_prop}. Then we have that for $\vert z \vert \geq \frac{C_2}{2T}$,
\[ h = T^2 \bigg( \frac{1}{1+(Tz)^2} + (2\sigma_0(z)-1) \pi \psi_0(p_D)  \frac{z+\frac{T^2z^3}{3}}{(1+(Tz)^2)^2}+ \mathcal{O}(T^{-1}) \bigg).\]
 Now for each $T$ let $\tilde{h}$ be a smooth function satisfying
\[ T^{-2} \tilde{h} = \begin{cases} T^{-2} h  &  \vert z \vert  < \frac{C_2}{2T}\\
T^{-2} h + \mathcal{O}(T^{-1})  &  \frac{C_2}{2T} \leq \vert z \vert < \frac{C_2}{T} \\
\frac{k_\pm z+1}{\frac{2}{3}T^2k_\pm z^3 + T^2 z^2 + 1} + \mathcal{O}(T^{-5}z^{-4}) & \frac{C_2}{T} \leq \vert z \vert < 1
\end{cases}\]
with $k_\pm$ interpreted as $k_-$ for $z< 0 $ and $k_+$ for $z > 0$.  This is possible because for $\frac{C_2}{2T} \leq \vert z \vert \leq \frac{C_2}{T}$, 
\[  \frac{k_\pm z+1}{\frac{2}{3} T^2k_\pm z^3 + T^2 z^2 + 1} = \frac{1}{1+(Tz)^2}  \pm \pi  \psi_0(p_D) \frac{z+\frac{T^2z^3}{3}}{(1+(Tz)^2)^2}+ \mathcal{O}(T^{-1}) .\]
 Therefore a smooth interpolation of $T^{-2}h$ and $ \frac{k_\pm z+1}{\frac{2}{3} T^2k_\pm z^3 + T^2 z^2 + 1}$ in this region differs from $T^{-2}h$ by a function that is $\mathcal{O}(T^{-1})$. But the zero eigenfunction $\psi_0$ is constant, so by Equation \ref{deltah_decay}, the $D$-dependent terms of $\delta h$ are $T\mathcal{O}((Tz)^{-4})$ for $\vert z \vert > \frac{C_2}{2T}$. Therefore after a $T^{-1} \mathcal{O}((Tz)^{-4})$ correction, we can take $T^{-2}(\tilde{h}-h)$ to be a function of $z$ only.

For simplicity of notation, the symbol $h$ will refer to the function  $\tilde{h}$ for the remainder of this paper.

\subsection{Expansion near the singular point}\label{expansionsNearSing}
Equation \ref{deltah_eqn} and the linearization of Equation \ref{maineqn2} yield a useful expansion of the metric near the singular point. 

First, we introduce a notation for describing the regularity of functions defined near a singularity. For $\ell \in  \mathbb{N}$ and $U \subseteq \mathbb{R}_{a,b,w}^3$ containing $0$, let \[r_w(a,b,w) = \sqrt{a^2 + b^2 + w^2}\] and 
\begin{equation}\label{Wspacedef}
\mathcal{W}^\ell(U) = \lbrace f \in C^{\infty}(U \setminus \lbrace 0 \rbrace): \ \nabla^j f = \mathcal{O}(r_w^{\ell-j} \log(r_w)) \ \text{as} \ r_w \rightarrow 0, \ \forall j  \geq 0\rbrace.
\end{equation}
Now for any $\ell \in \mathbb{N}$ and $U \subset \mathbb{R}_{a,b,w}^3$ containing $0$ we take $w^{\ell}$ to be a function in $W^{\ell}(U)$. Similarly, for any $k \in \mathbb{Z}_{\geq 0}$ and $\alpha \in (0,1)$, we take $g^{k,\alpha}$ to be a function in $C^{\infty}(U \setminus \lbrace p \rbrace) \cap C^{k,\alpha}(U)$. We allow $g^{k,\alpha}$, $w^{\ell}$, and the corresponding domains $U$ to change in each invocation (as is convention for the constant $C$), but it will always be assumed that they have no dependence on $T$. 

The following elementary result will provide the desired regularity for an expansion around the singular point.
\begin{lemma}\label{llemma}
Let $f(a,b,w)$ be a homogeneous polynomial of degree $k$ and $\mathcal{L}'$ a differential operator such that $\mathcal{L}'(\mathcal{W}^1) \subseteq \mathcal{W}^1$. Then there exist $u, \, v' \in \mathcal{W}^1$ such that 
\[ (\Delta_{\mathbb{R}^3}+\mathcal{L}') \, u = \frac{f}{r_w^{k+1}} + v'.  \] 
\end{lemma}
\begin{proof} Since $\mathcal{L}'(\mathcal{W}^1) \subseteq \mathcal{W}^1$, we may assume $\mathcal{L}' = 0$. Now let $p^{(k)}$ be a homogeneous polynomial of degree $k$. Then
\[ \Delta_{\mathbb{R}^3}\frac{p^{(k)}}{r_w^{k-1}} = -(k-1)(k+2)\frac{p^{(k)}}{r_w^{k+1}} +  \frac{\Delta_{\mathbb{R}^3} p^{(k)}}{r_w^{k-1}}.\]
But $\nabla (\frac{p^{(k)}}{r_w^{j}}) = \frac{p^{(k+1)}}{r_w^{j+2}}$ for $p^{(k+1)}$ a homogeneous polynomial of degree $k+1$, so $\frac{p^{(k)}}{r_w^{k-1}}  \in \mathcal{W}^1$. In particular, $\frac{f}{r_w^{k-1}} \in \mathcal{W}^1$, so taking $u = \frac{-1}{(k-1)(k+2)} \frac{f}{r_w^{k-1}}$, we see that if $k > 1$ and the result holds for $k' < k$, it holds for $k$ as well. 

To establish the base cases $k=0,1$, we note that 
\[ \Delta_{\mathbb{R}^3} r_w = \frac{2}{r_w} \]
and 
\[ \Delta_{\mathbb{R}^3}(a \log(r_w)) = \frac{3a}{r_w^2}.\] 
\end{proof}

Now return to the space constructed above and assume that in the notation of Section \ref{hcorrection}, $\vert z \vert < \frac{C_2}{2T}$. If we let
\[ \mathcal{L}_{T,z}^h = (\Delta_D + T^{-2} \partial_z^2) + z^2 \partial_z^2 + 6z \, \partial_z + 4\]
then Equation \ref{deltah_eqn} gives that 
\[   \mathcal{L}_{T,z}^h (\delta h) = 2\pi \delta_p.\]
 Define a new coordinate 
\begin{equation}\label{w_def}
w = Tz.
\end{equation}
Changing coordinates, we have that 
\[ \mathcal{L}_{T,z}^h =   (\Delta_D + \partial_w^2) + w^2 \, \partial_w^2 + 6w \, \partial_w + 4 = \Delta_{\text{cyl}} + w^2 \, \partial_w^2 + 6w \, \partial_w + 4 = \mathcal{L}_{1,w}^h,\]
where $\Delta_{\text{cyl}}$ is the Laplacian on the Riemannian product $D \times \mathbb{R}_w$,
and 
\begin{equation}\label{h_local_eqn}
\mathcal{L}_{1,w}^h (\delta h) = T \, 2\pi \delta_p,
\end{equation}
where $\delta_p$ now refers to a delta function with respect to the coordinate $w$. 

Now let $y = a+ib$ be  a K\"ahler normal coordinate on $D$ in a neighborhood $U$ of the singular point $p_D$ and define 
\begin{equation}\label{rw_def}
r_w = \sqrt{\vert y \vert^2 + w^2} = \sqrt{a^2 + b^2 + w^2}
\end{equation}
on $U$. Then if $\omega_0$ is the standard metric and $\Delta_0 = (\partial_a^2+\partial_b^2)$ is the standard Laplacian in these coordinates, \[\Delta_D - \Delta_0 = \mathcal{O}(r_w^2)\mathcal{D},\] where $\mathcal{D}$ is a second-order differential operator with continuous coefficients with respect to $y$ and $w$. Let $\mathcal{R} =   w^2 \partial_w^2 + 6 w \, \partial_w + 4$.  Since \[  \Delta_{\text{cyl}} \, (r_w^{-1}) = 4\pi  \delta_p \]  in the sense of distributions, we find that
\begin{dmath*} \mathcal{L}^h_{1,w} \bigg(\delta h -  \frac{T}{2r_w} \bigg)   =  \mathcal{L}^h_{1,w}(\delta h) - (\Delta_0 + \partial_w^2) \bigg(\frac{T}{2r_w}\bigg) - (\Delta_D - \Delta_0) \bigg(\frac{T}{2 r_w}\bigg) - \mathcal{R} \bigg(\frac{T}{2  r_w}\bigg) 
=  T g^{0,\alpha}(a,b,w)  + \frac{T p^3(a,b,w)}{r_w^4}  + \frac{T p^4(a,b,w)}{r_w^5}, 
\end{dmath*}
where $p^k(a,b,w)$ is a homogeneous polynomial of degree $k$ in $a,b, w$ with coefficients independent of $T$. Lemma \ref{llemma} plus Schauder regularity give that
\begin{equation}\label{deltah_expansion_singularpt} 
\delta h = \frac{ T}{2 r_w} + Tg^{2,\alpha} + T w^1.
\end{equation}

Equation \ref{deltah_expansion_singularpt} yields a similar expression for $\delta \chi$. By Equation  \ref{linearized_before_deriv}, we have that for some smooth function $g^{\infty}$ on $D$,
\begin{align*}
\delta \chi &= h_0^{-1} \delta h + 2 \int z \delta h \, dz + z g^{\infty}\\
&= \frac{w^2+1}{T^2}\bigg( \frac{ T}{2 r_w} + Tg^{2,\alpha} + T w^1 \bigg) + \frac{2}{T} \int \bigg( \frac{w}{2r_w} + w \, g^{2,\alpha} +  w\, w^1 \bigg) \, dw + z g^{\infty} \\
&= \frac{1}{2T r_w}  + \frac{r_w}{T} +  T^{-1} g^{2,\alpha} + T^{-1} w^1.
\end{align*}
Since $r_w \in \mathcal{W}^1$, this implies that 
\begin{equation}\label{chi_expansion} \delta \chi =  \frac{1}{2 T \, r_w} + T^{-1}g^{2,\alpha} + T^{-1} w^1.
\end{equation}

\begin{remark}\label{deltachi_D} We can repeat the above analysis on an arbitrary coordinate path $U$ on $D$. If $U$ does not contain the singular point, we find that 
\[ \delta h\vert_U = Tg^{2,\alpha} + T w^1.\]
Since $D$ is compact, this implies that 
\[ \delta h =  \frac{T}{2r_w} + T g^{2,\alpha} + T w^1 \ \ \text{ on }D \times [-C_2,C_2]_w.\]
Similarly, 
\[ \delta \chi =  \frac{1}{2Tr_w} + T^{-1}g^{2,\alpha} + T^{-1} w^1 \ \ \text{ on }D \times [-C_2,C_2]_w.\]

 \end{remark}

\section{Construction of the Singular $S^1$ Fibration Over $D \times [-1,1/2]$}\label{m_construction}
In Section \ref{kahler_reduction}, we derived a system of equations sufficient to define a K\"ahler-Einstein metric on an $S^1$ fibration over $D \times [-1,1/2] \setminus \lbrace p \rbrace$. In Section \ref{solution}, we constructed the data $(\chi,h) = (\chi^T,h^T)$ of a family of approximate solutions to this system parameterized by $T \gg 0$. As noted in Section \ref{derivation_of_system}, there is a gauge freedom in the choice of connection on the resulting $S^1$ fibration. Now we will make a choice of connection such that the $S^1$ fibration can be completed over the singular point to a $C^{2,\alpha}$ fibration $\mathcal{M}$, allowing us to define the family $\omega_T$ of approximately K\"ahler-Einstein metrics on $\mathcal{M}$. To motivate this choice, we first consider two spaces on which we aim to model the geometry of $\mathcal{M}$: the Taub-NUT space, to describe the behavior near $p$, and the Calabi model space, to describe the behavior as $z$ approahces the ends of the interval $[-1,1/2]$. Our techniques in this section follow \cite{sz19} Section 4.1.

\subsection{Two model spaces}
\subsubsection{The Taub-NUT space}\label{taubNUT}
Let $u_1, u_2$ be coordinates on $\mathbb{C}^2$ and $y, \, \bar{y}, \,  w$ coordinates on $\mathbb{R}^3 = \mathbb{C} \oplus \mathbb{R}$. The Hopf fibration $\pi_{H}: \mathbb{C}^2 \setminus  \lbrace (0,0)\rbrace \rightarrow \mathbb{R}^3 \setminus \lbrace (0,0,0) \rbrace$ is an $S^1$ fibration given by 
\[ \pi_H(u_1,u_2) = (u_1u_2, \frac{1}{2}(\vert u_1 \vert - \vert u_2 \vert^2)) = (y,w).\] It is checked that if $r = \sqrt{\vert y \vert^2 + w^2}$, then 
\[ g_{\mathbb{C}^2} = \frac{1}{2r} \pi_H^{\ast} g_{\mathbb{R}^3} + 2r\, \Theta_0^2,\]
where \[ \Theta_0 = \frac{1}{2r} J \pi^{\ast}_H dw,\]
and $J$ is the standard complex structure on $\mathbb{C}^2$ with the convention $Jdy = -i \, dy$. The Taub-NUT space is constructed by changing $\frac{1}{2r} \rightarrow \frac{1}{2r} + a$ for some constant $a > 0$. The resulting metric is still complete and Ricci flat. 
\begin{definition} Fix $a > 0$. The Taub-NUT space with parameter $a$ on $\mathbb{C}^2$ is
\[ g_{TN,a} = \bigg(\frac{1}{2r} + a\bigg) \pi_H^{\ast} g_{\mathbb{R}^3} +\bigg( \frac{1}{2r} + a\bigg)^{-1} \Theta^2, \]
where 
\begin{equation}\label{theta0_def}
\Theta = \bigg(\frac{1}{2r} + a\bigg) J \pi^{\ast}_H dw.
\end{equation} 
\end{definition} 
The Taub-NUT space is invariant under the $S^1$ action on $\mathbb{C}^2$ that the rotates the fibers of $\pi_H$, so as in the discussion in Section \ref{derivation_of_system}, the K\"ahler form of the Taub-NUT space $g_{TN,a}$ is
\begin{equation}\label{taubNutCoords} \omega_{TN,a} = \bigg(\frac{1}{2r_w} + a \bigg)\pi_H^{\ast}\omega_{\mathbb{C}} + dw \wedge \Theta,
\end{equation}
and
\begin{equation}\label{taubNUT_curvature} d\Theta = \partial_w \pi^{\ast} \omega_{0} - dw \wedge d^c \bigg(\frac{1}{2r}+ a\bigg),
\end{equation}
where $\omega_{0} = (\frac{1}{2r}+ a) \frac{i}{2}dy \wedge d\bar{y}$ and $d^c$ is computed with respect to the complex structure on $\mathbb{C} \subset \mathbb{C} \oplus \mathbb{R}$ (see \cite{sz19} 2.3). 

Note that if we make the change of coordinates $\underline{y} = \frac{a}{b} y$,  $\underline{w} = \frac{a}{b} w$, and $r_{\underline{w}} = \sqrt{\vert \underline{y} \vert^2 + \vert \underline{w} \vert^2}$, then we have
\begin{equation}\label{rescale_TN} a \, g_{TN,a} = b \, \underline{g}_{TN,b},
\end{equation}
where $\underline{g}_{TN,b}$ is the Taub-NUT metric with parameter $b$ with respect to the new coordinates.

\subsubsection{The Calabi model space}\label{calabimodel}Let $(D,\omega_D)$ be a compact curve with $K_D$ ample and $\omega_D$ negative K\"ahler-Einstein normalized such that $\omega_D \in 2\pi  \, c_1(K_D)$. We can choose a Hermitian metric $\Vert \cdot \Vert_D$ on $K_D$ such that 
\begin{equation}\label{d_ke}
- i\partial \bar{\partial} \log \Vert \cdot \Vert_D^2 = \omega_D.
\end{equation}  Our goal is to construct a negative K\"ahler-Einstein metric with constant $-1$ on a subset of the total space of $nK_D$ for some $n \in \mathbb{Q}$. We hypothesize that there exists such a metric of the form 
\begin{equation}\label{calabiansatz} \omega =  i \partial \bar{\partial} F(-\log \Vert \cdot \Vert_D^{2n})
\end{equation}
for some function $F$. We write $x = -\log \Vert \cdot \Vert_D^{2n}$ and view $F$ as a function of $x$. 

\begin{figure}[h]
\begin{tikzpicture}[scale=0.6]
\draw[thick] (0,0) ellipse (5cm and 1cm);
\draw[thick] (-7.5,7.5) to [out=310, in=90](-5,0);
\draw (-4.5,7) to [out=300, in=95](-3,2);
\draw (-1.75,7) to [out=290, in=90](-1,2);
\draw (1.75,7) to [out=250, in=90](1,2);
\draw (4.5,7) to [out=240, in=85](3,2);
\draw[thick] (7.5,7.5) to [out=230, in=90](5,0);

\draw[dashed] (-5.7,4.5) to [out=12, in=168](5.7,4.5);

\node[black] at (-9,0) {$(D,\omega_D)$};
\draw[->] (-9,3) to (-9,1);
\node[black] at (-9,4) {$(nK_D,\Vert \cdot \Vert_D^n)$};
\node[black] at (10,2.25) {$(T_D, i \partial \bar{\partial} F(-\Vert \cdot \Vert_D^{2n}))$};
\node at (-13,2.25){};
\node at (0,-1.5){};
\node at (0,9){};

\end{tikzpicture}
\caption[The Calabi model space]{The Calabi model space. Let $(D,\omega_D)$ be a compact negative K\"ahler-Einstein curve with canonical bundle $K_D$ and $\Vert \cdot \Vert_D$ a Hermitian metric on $K_D$ with curvature $\omega_D$. The Calabi model space is the tubular neighborhood $T_D = \lbrace -1/2n < z < 0 \rbrace \subset nK_D$ equipped with the metric $\omega = i \partial \bar{\partial} F(-\Vert \cdot \Vert_D^{2n})$. }
 \label{f: construction}
\end{figure}
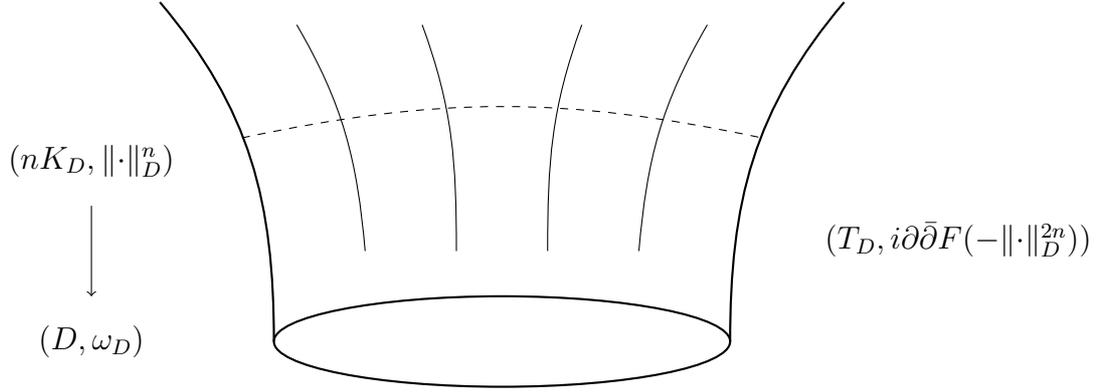

If locally $\Vert \cdot \Vert_D^2 = h\vert u \vert^2$ for a real function $h(y)$ of a local $D$ coordinate $y$ and fiber coordinate $u$, then by Equation \ref{d_ke},
\[ -i\partial \bar{\partial}\log \left(\frac{-i\partial \bar{\partial} \log(h)}{\frac{i}{2} dy \wedge d\bar{y}} \right)=  i \partial\bar{\partial}\log(h),\] which will be satisfied if
\[ -\frac{h\partial \bar{\partial}h -  \partial h \wedge \bar{\partial}h}{h^2} =  \frac{i}{2h} dy \wedge d\bar{y},\]
or equivalently
\begin{equation}\label{hidentity} i \partial \bar{\partial} x = n \frac{i}{2h} dy \wedge d\bar{y}.
\end{equation} 
Taking the Ricci curvature of $\omega$ and applying Equation \ref{hidentity}, we find that it is sufficient to solve the ordinary differential equation
\begin{equation}\label{calabiode} \log(F' F'')+\frac{x}{n}- F = C
\end{equation}
for any constant $C$. 

The generator of the natural $S^1$ action on $nK_D$ is
\[ \xi = i(u \, \partial_u - \bar{u} \, \partial_{\bar{u}}),\] 
so if a metric is defined by Equation \ref{calabiansatz}, 
\[ h^{-1} = \Vert \xi \Vert_{\omega}^{2} = 2\frac{d^2F}{d x^2}.\] 
In addition, by Cartan's formula, $\iota_{\xi}\omega = i \, d  \, \xi \lrcorner \bar{\partial} F(x),$
so
\[ z(x) =  \frac{dF}{dx}-\frac{1}{n}\]
is a moment map for such a metric. Let \[\mu(z) = \frac{nz(x)^3}{3}+\frac{z(x)^2}{2}\] and define \[F(x) = \log(\mu(z(x)))+\frac{x}{n}.\] Then
\[ z(x) =  \bigg(\frac{\mu'}{\mu} \bigg) \frac{dz}{dx}, \]
while 
\[ h^{-1} = 2 \frac{dz}{dx} = 2\bigg(\frac{z \mu}{\mu'}\bigg),\]
and it is easily verified that $F$ satifies Equation \ref{calabiode}.

By the discussion in Section \ref{kahler_reduction}, $\omega$ satisfies Equation \ref{omega_defn}. A computation with Equations \ref{omega_defn} and  \ref{calabiansatz} then gives that $\tilde{\omega} = (1+nz) \omega_D$.

Note that \[ x = \int^z \frac{ns + 1}{\frac{ns^3}{3} + \frac{s^2}{2}} \, ds + C.\]
The integrand is positive for $-1/n < z < 0$. Restricting $z$ to this interval, we see that since $x \rightarrow \infty$ near the zero section, $z$ increases towards the zero section as well. We define the Calabi model space as the subset \[ T_D = \lbrace -1/2n < z < 0 \rbrace \subset nK_D\] equipped with the metric $\omega$. This space is punctured along the zero section and has a ``tubular'' boundary. 

For convenience of notation we consider the $n = 0$ case as part of this family. In this case the metric takes the $D$-invariant form described in Equation \ref{ansatz} and we define $T_D = \lbrace -1 < z < 0 \rbrace$.

%

\subsection{Construction of $\mathcal{M}^{\ast}$}\label{hopf_identification}
Our neck region $\mathcal{M}$ should resemble a Calabi model space at each end, but the two Calabi model spaces have different degrees and are oriented in opposite directions. Concretely, if $\Theta = h J dz$ as in Equation \ref{kahlerCondition} and taking $k_\pm$ as in the statement of Theorem \ref{mainthm}, near $\partial \mathcal{M}_-$ we must have that \[h = \frac{k_- z + 1}{\frac{2}{3}k_- z^3 + z^2 + T^{-2}}\] while near $\partial \mathcal{M}_+$, we must have that \[h = \frac{k_+ z + 1}{\frac{2}{3}k_+ z^3 + z^2 + T^{-2}}.\] Note that if we interpret the negative sign of $k_+$ as part of the $z$ coordinate, $\Theta$ behaves as if $z$ increases towards the zero section, as in our discussion of the Calabi model space. This explains why $k_- \geq 0$ and $k_+ \leq 0$. 

Recall the $2$-form $\Gamma = \partial_z \tilde{\omega} - dz \wedge d^c h$ on $D \times [-1,1/2] \setminus \lbrace p \rbrace$. By Equation \ref{kahlerCondition}, $d\Theta = \Gamma$, so $\Gamma$ is closed. By Lemma \ref{gammaIntCoho}, $\frac{1}{2\pi}\Gamma$ is the curvature of a connection form of a singular $S^1$ bundle over $D \times [-1,1/2]$. This $S^1$ bundle will become the manifold $\mathcal{M}$ underlying our family of metrics, and the change from $k_-$ to $k_+$ will correspond to a change in the degree of its restriction to $D$. 

In Section \ref{backwards} we made the choice to add a singularity of the form $2\pi\delta_p$ to Equation \ref{maineqn2}. We can now see that this singularity increments the degree of the circle bundle corresponding to $\Gamma$ by one, corresponding to the assumption that $k_- -k_+   = 1$. Larger increments between $k_-$ and $-k_+$ are achieved by adding multiple separate singular points to the neck region, for instance by choosing several points $p_1,...,p_n$ on $D \times \lbrace 0 \rbrace$ and changing the inhomogeneous term to $2 \pi \sum_{i=1}^n \delta_{p_i}$. We then add the linearized solutions constructed in Section \ref{solution} and the analysis goes through without change.

\begin{lemma}\label{gammaIntCoho}
For all $T > 0$, $\frac{1}{2 \pi} \Gamma \in H^2(D \times [-1,1/2] \setminus \lbrace p \rbrace;\mathbb{R})$ is integral.
\end{lemma} 
\begin{proof} 
By a Mayer-Vietoris argument it suffices to show that $\frac{1}{2 \pi} \Gamma$ integrates to an integer over a slice $D \times \lbrace z_0 \rbrace$ for some $z_0 \neq 0$ and over the boundary of a ball \[ S_{\epsilon}(p) = \lbrace q: \ r_w(q) \leq \epsilon \rbrace\] for some sufficiently small $\epsilon > 0 $.

For the first integral, note that since 
\[ \partial_z^2 \tilde{\omega} = -(\Delta_D h) \, \omega_D,\]
away from $z = 0$, we have that
\[ \partial_z [\partial_z \tilde{\omega}] =  [\partial_z^2 \tilde{\omega}] = 0 \in H^2(D;\mathbb{R}),\]
and therefore $[\partial_z \tilde{\omega}]$ is constant in $z$.  However, by Proposition \ref{deltah_decay_prop},
\[ \partial_z \tilde{\omega} = (k_{\pm} + \mathcal{O}(T^{-3})) \, \omega_D\]
for large $z$. Thus letting $\vert z \vert \rightarrow \infty$ (see Remark \ref{soln_extension}) shows that $[\partial_z \tilde{\omega}] = k_{\pm}[\omega_D]$. The component $dz \wedge d^c h$ does not contribute to the integral of $\Gamma$ over $D$. 

For the second integral, observe that by Stokes' theorem, 
\[ \int_{\partial S_{\epsilon}(p)}\frac{1}{2 \pi}\Gamma = \lim \limits_{\epsilon' \rightarrow 0}  \int_{\partial S_{\epsilon'}(p)} \frac{1}{2 \pi}\Gamma.\]
Now by Equation \ref{chi_expansion} we have for $r_w$ small that
\[ \partial_z \tilde{\omega} = \left(-\frac{Tz}{4r_w^3}\right)  i \, dy \wedge d\bar{y} + \mathcal{O}(\log(r_w))\]
and by Equation \ref{deltah_expansion_singularpt} we have that
\[ dz \wedge d^c h = \frac{-iT}{4r_w^3} dz \wedge (y \, d\bar{y} - \bar{y} \, dy) + \mathcal{O}(\log(r_w)).\]
But
\[ d(z \,  dy \wedge d\bar{y}  - y \, dz \wedge d\bar{y} + \bar{y} \, dz \wedge dy) = 3\, dz \wedge dy \wedge d\bar{y},\]
so 
\begin{dmath*} \lim \limits_{\epsilon' \rightarrow 0} \int_{\partial S_{\epsilon'}(p)} \frac{1}{2 \pi}\Gamma 
= \lim \limits_{\epsilon' \rightarrow 0} \int_{S_{\epsilon'}(p)} -\frac{3iT}{8\pi (\epsilon')^3} dz \wedge dy \wedge d\bar{y} + \mathcal{O}((\epsilon')^2\log(\epsilon'))
= -\lim \limits_{\epsilon' \rightarrow 0} \int_{S_{\epsilon'}(p)} \frac{3}{4 \pi (\epsilon')^3} \bigg(\frac{i}{2} \, dy \wedge d\bar{y} \bigg) \wedge dw
= -1,
\end{dmath*}
as desired.
\end{proof}

By Lemma \ref{gammaIntCoho}, there is an $S^1$ bundle $\pi: \mathcal{M}^{\ast} \rightarrow D \times [-1,1/2] \setminus \lbrace p \rbrace$ with connection form $-i\Theta'$ whose curvature is $-i\Gamma$. The connection form $-i\Theta'$ is not unique, but we fix an arbitrary choice. Note that the spaces $\mathcal{M}^{\ast}$ are diffeomorphic for all $T$, since changing $T$ only rescales all data with respect to the $z$ coordinate. 

\subsection{Compactification of $\mathcal{M}^{\ast}$}\label{compactification} Now we can construct a compactification of $\mathcal{M}^{\ast}$ modeled on the Taub-NUT space. As manifolds this only involves adding a point to $\mathcal{M}^{\ast}$ and it is easily seen that the resulting spaces will be diffeomorphic for all $T$. Let $(y,\bar{y},w)$ be the coordinates on $D \times [-1,1/2]$ in the punctured neighborhood $U' = U \times [-1,1/2] \setminus \lbrace p \rbrace$, where $U$ is as defined in Section \ref{expansionsNearSing}. If $V = \pi^{-1}(U')$, then in some coordinates $u_1,u_2$, $V$ gives a subset of $\mathbb{C}^2\setminus \lbrace 0 \rbrace$, and we can define the Hopf fibration $\pi': V \rightarrow U'$ by
\begin{equation}\label{hopf} (u_1,u_2) \rightarrow (u_1 u_2, \frac{1}{2}(\vert u_1 \vert^2 - \vert u_2 \vert^2)).
\end{equation}
Now $\pi$ and $\pi'$ are $S^1$ fibrations of the same degree, so there is an $S^1$-equivariant smooth map $\phi: V \rightarrow \mathcal{M}$, diffeomorphic with its image, such that $\phi^{\ast} \pi = \pi'$. In other words, we can choose complex-valued (though not necessarily holomorphic) coordinates $u_1$ and $u_2$ such that $\pi$ is the Hopf fibration given in coordinates by Equation \ref{hopf}, i.e. $y = u_1 u_2$, $w = \frac{1}{2}(\vert u_1 \vert^2 - \vert u_2 \vert^2)$, and the connection is given by $\phi^{\ast} (-i\Theta')$. 

From now on we will assume we have made such a transformation and take $u_1$, $u_2$, and $\pi$ as in the Hopf fibration. In addition, we define in these coordinates
\begin{equation}\label{sdef} 
s^2 = \vert u_1 \vert^2 + \vert u_2 \vert^2.
\end{equation}
Notice that 
\begin{equation}\label{s_vs_r}
\pi^{\ast}r_w = \frac{1}{2} s^2.
\end{equation}

The connection form determines the metric on $\mathcal{M}^{\ast}$ via Equation \ref{omega_defn}. However, our arbitrary choice of connection may not have the desired behavior near $p$. Now we modify $\mathcal{M}^{\ast}$ by a Gauge transformation so that near $p$ the metric on $\mathcal{M}^{\ast}$ differs from the Taub-NUT space by a form of sufficiently high regularity. This will give a $C^{2,\alpha}$ compactification of $\mathcal{M}^{\ast}$ over $p$, completing our construction of the singular $S^1$ fibration $\mathcal{M}$ and its $C^{2,\alpha}$ K\"ahler structure. 

We adopt the coordinates $u_1$, $u_2$ defined above so that the map $\pi: \mathcal{M}^{\ast} \rightarrow D \times [-1,1/2] \setminus \lbrace p \rbrace$ is identified with the Hopf fibration. Define $\Theta_0$ by Equation \ref{theta0_def}, taking $a = T$, $J$ the natural complex structure on $\mathbb{C}_{u_1,u_2}^2$, and with $\pi_H$ given by $\pi$. Note also that by our choice of coordinates in Section \ref{hopf_identification}, $r = r_w$. By Equations \ref{deltah_expansion_singularpt}, \ref{chi_expansion}, and  \ref{taubNUT_curvature}, we have that near $p \in D \times [-1,1/2]$,
\begin{dmath*}
\Gamma- d \Theta_0 = \partial_z(\tilde{\omega} - T^{-1}\omega_{0}) + dw \wedge d^c_D  \bigg(  \frac{h}{T}-\frac{1}{2r_w} -T \bigg)
= \partial_z (g^{2,\alpha} + w^1) \frac{i}{2} dy \wedge d\bar{y} + dw \wedge d^c_D (g^{2,\alpha}+w^1)
= c_1 \, dy \wedge d\bar{y} + c_2 \, dw \wedge dy + c_3 \, dw \wedge d\bar{y}
\end{dmath*}
for some $c_1 \in C^{1,\alpha}(U) + W^0(U)$ and $c_2, c_3 \in C^{2,\alpha}(U) + W^1(U)$. Pulling back to $\mathcal{M}^{\ast}$, we have that with respect to the flat metric in $u_1, \, u_2$,
\begin{equation}\label{c1_regularity}
\pi^{\ast} (c_1 \, dy \wedge d\bar{y}) =  (g^{1,\alpha} + w^{0}) \, \mathcal{O}(s^2) = g^{1,\alpha},
\end{equation}
while by the same argument the $c_2$ and $c_3$ term have higher regularity. Therefore we denote $\omega^{1,\alpha} = c_1 \, dy \wedge d\bar{y} + c_2 \, dw \wedge dy + c_3 \, dw \wedge d\bar{y}$. 

If we impose the gauge-fixing condition $d^{\ast} \theta = 0$, then for some $\epsilon > 0$, we can solve the elliptic system 
\[ \left\lbrace \begin{array}{l} d\theta = \omega^{1,\alpha} \\
d^{\ast} \theta = 0 \\ 
\theta(\nu) = 0 \text{ on } \delta B_{\epsilon}(p) \end{array} \right. \] 
on $B_{\epsilon}(p)$, where $\nu$ is the unit normal to $B_{\epsilon}(p)$. The resulting one-form $\theta$ is smooth away from $\pi^{-1}p$ since $\omega^{1,\alpha}$ is, and since the system is $S^1$ invariant, $\theta$ can be chosen $S^1$-invariant by averaging. We also have that $\theta \in C^{2,\alpha}(B_{\epsilon}(p))$, since 
\[ \Delta \theta = d^{\ast} \omega^{1,\alpha} \in C^{0,\alpha}(B_\epsilon(p)). \]
Finally, $\theta(\partial_t) = 0$. For by Cartan's formula, $d(\iota_{\partial_t}\theta) = -\iota_{\partial_t} \omega^{1,\alpha} = 0$, so $\theta(\partial_t)$ is constant. But $\partial_t \rightarrow 0$ near $p$, so by the regularity of $\theta$, $\theta(\partial_t) = 0$.

Now define $\Theta =  \Theta_0 + \theta$. Because $d\Theta = \Gamma = d\Theta'$, there is a gauge transformation that takes $\Theta'$ to $\Theta$. Since gauge transformations only rotate the fibers of the $S^1$ bundle, $\pi$ is still modelled on the Hopf fibration as in Equation \ref{hopf}. However, we will see in Section \ref{rescaledGeo} by taking $\Theta$ as our connection, $\omega$ is asymptotic to the Taub-NUT metric near $p$. Therefore $(\mathcal{M},\omega)$ is a $C^{2,\alpha}$ compactification of $\mathcal{M}^{\ast}$.

\section{Limiting Behavior of the Approximate Solution} \label{schauderestimatesection}
\subsection{Convergence of Riemannian manifolds} 
The space $(\mathcal{M},\omega_T)$ we constructed in Sections \ref{solution} and \ref{m_construction} is K\"ahler, but since it solves the linearization of Equation \ref{maineqn1} rather than the full equation, it is only approximately Einstein. We aim to argue that for large enough $T$, $\omega_T$ is sufficiently close to being Einstein that it can be perturbed to an Einstein metric. Several notions of the distance between metrics will be useful in making this argument. Our discussion in this section follows \cite{Petersen} Chapter 11. 

\subsubsection{Gromov-Hausdorff convergence}
The weakeast notion of convergence we use is Gromov-Hausdorff convergence. The Gromov-Hausdorff distance quantifies the dissimilarity between metric spaces. 

\begin{definition}Let $(X,d_X)$ and $(Y,d_Y)$ be metric spaces and let $\mathcal{A}$ be the set of metrics on $X \cup Y$ that extend $d_X$ and $d_Y$. The Gromov-Hausdorff distance is defined as
\[ d_{\text{GH}}((X,d_X),(Y,d_Y)) = \text{inf}_{d \in \mathcal{A}} \, d(X,Y),  \]
where for a metric $d$ on $X \cup Y$, $d(X,Y) = \text{inf} \lbrace \epsilon: \forall x \in X, \exists y \in Y: d(x,y) < \epsilon \rbrace$. 
\end{definition}

In Riemannian geometry, the metric spaces will be Riemannian manifolds $(M,g)$ with the metric $d_g$ induced by distance in the Riemannian metric $g$. Thus we say that $(M_i,g_i)$ converges to $(M,g)$ in the Gromov-Hausdorff topology if \[d_{\text{GH}}((M_i,d_{g_i}),(M,d_g)) \rightarrow 0\] as $i \rightarrow \infty$. In such a situation, we may also informally say that $x_i \in M_i \rightarrow x \in M$ if there exists a sequence of metrics $d_i \in \mathcal{A}$ realizing the Gromov-Hausdorff convergence such that $d(x_i,x) \rightarrow 0$.

The Gromov-Hausdorff distance defines a complete, separable metric space on the set of equivalence classes under isometry of compact metric spaces (see \cite{Petersen} 11.1.18). On noncompact spaces, we instead consider pointed Gromov-Hausdorff convergence. Let $(X,d_X,x)$ and $(Y,d_Y,y)$ be metric spaces with distinguished points $x$ and $y$. We define 
\[ d_{\text{GH}}((X,d_X,x),(Y,d_Y,y)) = \text{inf}_{d \in \mathcal{A}} \,  (d(X,Y) + d(x,y)).  \]
Then $(X_i,d_{X_i},x_i)\rightarrow (X,d_X,x)$ if for all $R > 0$, $(\bar{B}_R(x_i),d_{X_i},x_i) \rightarrow (\bar{B}_R(x),d_X,x)$ with respect to the pointed Gromov-Hausdorff distance. 

Gromov-Hausdorff convergence on its own is a relatively weak notion. It does not imply the convergence of derivatives in any sense, and limits of manifolds may not be manifolds. In addition, a sequence of $n$-dimensional manifolds may converge to a manifold of any lower dimension (by collapsing) or higher dimension (by space-filling). With additional assumptions, however, Gromov-Hausdorff convergence can imply stronger convergence. For instance, Cheeger and Naber \cite{CN} proved that a sequence of manifolds with Ricci curvature uniformly bounded and local volume noncollapsing converges smoothly outside a singular set of real codimension at least $4$. 

The local volume noncollapsing assumption does not hold in our case, and we will see that if we rescale to a constant diameter, our spaces converge to a real interval. The usefulness of Gromov-Hausdorff convergence for our purposes is in allowing us to describe the limiting geometry of different parts of the neck region we are constructing. We will see that after an appropriate rescaling, some portions of the neck region collapse to lower-dimensional spaces. The notion of Gromov-Hausdorff convergence allows us to state the sense in which these lower-dimensional spaces occur as limits. Subsequent analysis will require passing to local universal covers to ``un-collapse'' these spaces and achieve $C^{k,\alpha}$ convergence.

\subsubsection{H\"older regularity scales}\label{holderGeneralSection}
It is often useful to discuss convergence of the derivatives of a sequence of metrics. For this purpose we measure the H\"older norms of these metrics in local coordinates. If we were to naively allow any choice of charts, however, we could ``zoom in'' to normal neighborhoods around each point and every smooth metric would look like the Euclidean metric. Therefore we must control the radius of the chart and the extent of magnification in our definition of the H\"older norm.

\begin{definition}\label{holder_general} Let $(M,g)$ be a Riemannian manifold, $p \in M$ a distinguished point, and $r > 0$. For any $k \in \mathbb{Z}_{\geq 0}$ and $\alpha \in (0,1)$, $\Vert (M,g,p) \Vert_{C^{k,\alpha},r}$ is defined as the supremum over constants $Q$ such that there exists a $C^{k+1,\alpha}$ chart $\phi: (B_r(0),0) \subset \mathbb{R}^n \rightarrow (U,p) \subset M$ satisfying the following conditions:
\begin{enumerate}
\item $\vert D\phi \vert \leq e^{Q}$ on $B_r(0)$ and $\vert D\phi^{-1} \vert \leq e^{Q}$ on $U$.
\item For all multi-indices $I$ such that $\vert I \vert \leq k$, \[r^{\vert I \vert + \alpha} \Vert \partial^I g_{lm} \Vert_{C^{0,\alpha}(B_r(0))} \leq Q.\] 
\end{enumerate}
Then globally 
\[ \Vert (M,g) \Vert_{C^{k,\alpha},r} = \sup_{p \in M} \Vert (M,g,p) \Vert_{C^{k,\alpha},r}.\]
\end{definition} 
%

The H\"older norm computes the largest constant for which the metric can be controlled at a fixed scale $r$. Conversely, we can ask for the largest scale at which a fixed constant $Q$ controls the metric. This gives rise to the notion of a local regularity scale.

\begin{definition}\label{generalRegScale}
Fix $\epsilon > 0$ and $r > 0$. We say that a Riemannian manifold $(M,g)$ is $(r,k+\alpha,\epsilon)$-regular at $p \in M$ if $g$ is $C^{k,\alpha}$-regular on $B_{2r}(p)$ and
\[ \Vert (M,g,p) \Vert_{C^{k,\alpha},r} \leq \epsilon.\] 
We define the $C^{k,\alpha}$-regularity scale of $(M,g)$ at $p$ to be the supremum over the set of $r$ for which $(M,g)$ is $(r,k+\alpha,\epsilon)$-regular at $p$. 
\end{definition} 
Intuitevely, the $C^{k,\alpha}$ $\epsilon$-regularity scale at $p$ is the scale on which the $C^{k,\alpha}$ geometry of $M$ is ``$\epsilon$-interesting.'' Thus it provides a choice of resolution that is in a sense uniform throughout $M$. 


\subsubsection{H\"older convergence and Einstein regularity}\label{holderConvGeneral}

The $C^{k,\alpha},r$ norm defined in Section \ref{holderConvGeneral} is useful for deriving compactness results but cannot be used to measure the distance between two metrics. For this purpose, we make use of the $C^{k,\alpha}$ topology on compact Riemannian manifolds. Instead of defining the norm locally, we use the Riemannian distance function to compute the denominator in the H\"older seminorms. Then we say that $(M_i,g_i,p_i) \rightarrow (M,g,p)$ in the pointed $C^{k,\alpha}$ topology if there exists a compact exhaustion $U_j$ of $M$ and a collection $F_{j,i}: U_j \rightarrow M_i$ of diffeomorphisms mapping $p$ to $p_i$ such that $F_{j,i}^{\ast}g_i \rightarrow g$ in the $C^{k,\alpha}$ topology on $U_j$. 

Our use of the implicit function theorem to correct $(\mathcal{M},\omega_T)$ to a K\"ahler-Einstein metric will only give $C^{0,\alpha}$ convergence. Happily, the following result of Anderson and Colding gives higher regularity without any additional work.

\begin{theorem}\label{einsteinRegularity}  \textup{\cite{Anderson, Colding}} Let $(M_i,g_i)_{i \in \mathbb{Z}}$ and $(M,g)$ be compact $n$-dimensional Riemannian manifolds such that $\text{Ric }g_i = \lambda_i g_i$ for $\vert \lambda_i \vert \leq n-1$. If $(M_i,g_i) \rightarrow (M,g)$ in the Gromov-Hausdorff topology, then $(M_i,g_i) \rightarrow (M,g)$ in $C^{k,\alpha}$ for any $k$ and $\alpha \in (0,1)$. 
\end{theorem} 

\subsubsection{Regularity on local universal covers}\label{universalCoversSection}
We will see that Definition \ref{generalRegScale} and the notion of $C^{k,\alpha}$ convergence in Section \ref{holderConvGeneral} are too demanding for our family of spaces $(\mathcal{M},\omega_T)$. This is because as $T \rightarrow \infty$, the size of the $S^1$ fiber, given by $h^{-1}$, collapses near the singular point, so charts in this region must be correspondingly small. However, the curvature is still uniformly bounded, and if we unroll the $S^1$ fiber by taking the universal cover in a neighborhood of a point that does not contain $p$, we will find that the regularity scale is bounded from below. Using this observation, we generalize the idea of $C^{k,\alpha}$ convergence. These definitions follow \cite{sz19} 4.3. 

\begin{definition}\label{spaces_def}
Let $(M_i,g_i)$ be a sequence of Riemmanian manifolds of dimension $n$. For each $i$, a local universal cover of $(M_i,g_i)$ is the Riemannian universal cover $(\widetilde{B}_r(x_i),\tilde{g}_i)$ of $B_r(x_i) \subset M_i$ for any $x_i \in M_i$ and $r > 0$. If $(M_i,g_i)$ converges in the Gromov-Hausdorff topology to $(M,g)$ a manifold such that $\text{dim }M = n - 1$, we say that $(M_i,g_i)$ \textit{converges to $(M \times \mathbb{R},g \times g_{\mathbb{R}})$ in $C^{k,\alpha}$ on local universal covers} if whenever $x_i \rightarrow x \in M$, there exists $r > 0$ such that $\widetilde{B}_r(x_j) \rightarrow B_r(x) \times \mathbb{R}$ in $C^{k,\alpha}$. 
\end{definition}

We can also update Definition \ref{generalRegScale} (see \cite{sz19} Definition 4.22). 
\begin{definition}\label{lucRegScale}
Fix $\epsilon > 0$ and $r > 0$. We say that a Riemannian manifold $(M,g)$ is $(r,k+\alpha,\epsilon)$-regular in the sense of universal covers at $p \in M$ if $g$ is $C^{k,\alpha}$-regular on $B_{2r}(p)$ and
\[ \Vert (\tilde{B}_{2r}(\tilde{p}),\tilde{g},\tilde{p}) \Vert_{C^{k,\alpha},r} \leq \epsilon,\] 
where $(\tilde{B}_{2r}(\tilde{p}),\tilde{g})$ is the Riemannian universal cover of $B_{2r}(p)$ and $\tilde{p}$ is a preimage of $p$. We define the $C^{k,\alpha}$-regularity scale in the sense of universal covers of $(M,g)$ at $p$ to be the supremum over the set of $r$ for which $(M,g)$ is $(r,k+\alpha,\epsilon)$-regular in the sense of universal covers at $p$. 
\end{definition} 
For the remainder of this paper, when we discuss regularity scales, we mean ``in the sense of universal covers.'' However, when discussing convergence we will be explicit about whether we are referring to local universal covers or not. 
%

\subsection{Spaces of functions on $(\mathcal{M},\omega_T)$}\label{weightedSchauderGeneral}
We return to our goal of perturbing $(\mathcal{M},\omega_T)$ to an exact solution for large enough $T$. To achieve this, we must first prove a certain weighted Schauder estimate that is uniform in $T$. While in the previous section we discussed H\"older norms on families of manifolds, in this section we define H\"older spaces of functions on such families. The estimates we derive will establish the boundedness of the inverse in our use of the inverse function theorem in Section \ref{perturbation}. 

The global metric behavior of $(\mathcal{M},g_T)$ as $T \rightarrow \infty$ is not simple to describe. However, we will only need to understand the geometry near a sequence of points $(x_j)$ where $x_j \in (\mathcal{M},g_{T_j})$ for some sequence $T_j \rightarrow \infty$. We will see that after passing to a subsequence and rescaling, a sequence of neighborhoods of these points converges in the Gromov-Hausdorff sense to one of four model spaces. Further, we can achieve $C^{2,\alpha}$ convergence by passing to the local universal cover. Therefore we will see that the existence of a weighted Schauder estimate on $(\mathcal{M},g_T)$ reduces to a collection of statements about these model spaces. 

To state the desired Schauder estimate, we must first define the weight function $\rho^{(k+\alpha)}_{\delta,\nu,\mu}: \mathcal{M} \rightarrow \mathbb{R}_{> 0}$. In what follows, the functions $W$ and $\rho^{(k+\alpha)}_{\delta,\nu,\mu}$ depend on $T$, but we suppress this dependence for (relative) ease of notation. Pick $C_3 > 0$ such that $r_w$ is defined up to $C_3$ for all $T$ and let $W$ be a smooth and ``reasonable'' function such that
\begin{equation}\label{Wdef}
 W(q,T) = \begin{cases}
T^{-1} & r_w \leq T^{-1}  \\ 
r_w(q) & 2T^{-1} \leq r_w \leq \frac{C_3}{2} \\
1 & r_w \geq C_3 \text{ or $r_w$ is undefined.}
\end{cases}
\end{equation}
The significance of $W$ is that we will need to rescale $g_T$ by $W(q,T)$ to see nontrivial Gromov-Hausdorff convergence in a neighborhood of $q$. 

Now for $k \in \lbrace 0,1,2\rbrace$ define
\[ \rho^{(k+\alpha)}_{\delta,\nu,\mu} (q) = (1+w(q))^{-\delta} W(q,T)^{\nu + k  + \alpha} T^{\mu}.\]
The factor of $W(q,T)^{k  + \alpha}$ will ensure that the weight function is compatible with differentiation in the proof of Proposition \ref{schauderLocal}. The factor of $W(q,T)^{\nu}$ controls the behavior of functions at infinity from the vantage point of the singular point $p$. The factor of $(1+w(q))^{-\delta}$ controls the behavior of functions at infinity from the vantage point of the two boundary components. Finally, if $\mu$ is large enough then the factor of $T^{\mu}$ ensures that $\rho^{(0)}_{\delta,\nu+2,\mu}$ is bounded below, allowing us to take powers within the weighted H\"older space. 

\begin{figure}[h]
\begin{tikzpicture}[scale=0.6]
\draw[thick] (-10,5) to (10,5);
\draw[gray, dashed] (-10,7) to (10,7);
\draw[gray, dashed] (-10,5.5) to (10,5.5);
\draw[gray, dashed] (-10,9) to (10,9);

\draw (-10,5.5) to [out=5, in=200](1,8.7);
\draw (1,8.7) to [out=20, in =180](3,9);
\draw (3,9) to [out=0,in=140](7,7.5);
\draw (7,7.5) to [out=-40,in=180](8,7);
\draw (8,7) to (10,7);

\node[black] at (-11,9) {$T^{\mu}$};
\node[black] at (-11,7) {$T^{\mu-(\nu+2)}$};
\node[black] at (-11,5.5) {$T^{\mu-\delta}$};
\node[black] at (0,10) {Through the singular point};

\node[black] at (0,4) {$z$};
\node[black, style={rotate=90}] at (-13,7) {$\rho_{\delta,\nu,\mu}^{(2)}(p_D,z)$};

\node[black] at (-10,4.5) {$-1$};
\draw (-10,4.85) to (-10,5.15);
\node[black] at (8,4.5) {$-T^{-2}$};
\draw (8,4.85) to (8,5.15);
\node[black] at (3,4.5) {$-T^{-1}$};
\draw (3,4.85) to (3,5.15);
\node[black] at (10,4.5) {$0$};
\draw (10,4.85) to (10,5.15);


\draw[thick] (-10,-3) to (10,-3);
\draw[gray, dashed] (-10,-2.5) to (10,-2.5);
\draw[gray, dashed] (-10,1) to (10,1);

\draw (-10,-2.5) to [out=5, in=200](1,0.7);
\draw (1,0.7) to [out=20, in =180](3,1);
\draw (3,1) to (10,1);

\node[black] at (-11,1) {$T^{\mu}$};
\node[black] at (-11,-2.5) {$T^{\mu-\delta}$};
\node[black] at (0,2) {Away from the singular point};

\node[black] at (0,-4) {$z$};
\node[black, style={rotate=90}] at (-13,-1) {$\rho_{\delta,\nu,\mu}^{(2)}(q,z)$} ;

\node[black] at (-10,-3.5) {$-1$};
\draw (-10,-3.15) to (-10,-2.85);
\node[black] at (8,-3.5) {$-T^{-2}$};
\draw (8,-3.15) to (8,-2.85);
\node[black] at (3,-3.5) {$-T^{-1}$};
\draw (3,-3.15) to (3,-2.85);
\node[black] at (10,-3.5) {$0$};
\draw (10,-3.15) to (10,-2.85);

\end{tikzpicture}
\caption[The weight function $\rho^{(2)}_{\delta,\nu,\mu}$]{The weight function $\rho^{(2)}_{\delta,\nu,\mu}$ as a function of $z$ along two cross sections $(p,t) \in D \times I$ for $t \in (-1,0)$. In the first case, the cross section goes through the singular point $(p_D,0)$, while in the second case, it takes the form $\lbrace (q,t), t \in (-1,0) \rbrace$ for some point $q \in D$ such that $r_w((q,0)) > C_3$. Parameters are chosen in accordance with Theorem \ref{schauderGlobal}, and in addition we are assuming that $\nu + 2 < \delta$, though the opposite may be true.}
 \label{f: weightfxn}
\end{figure}
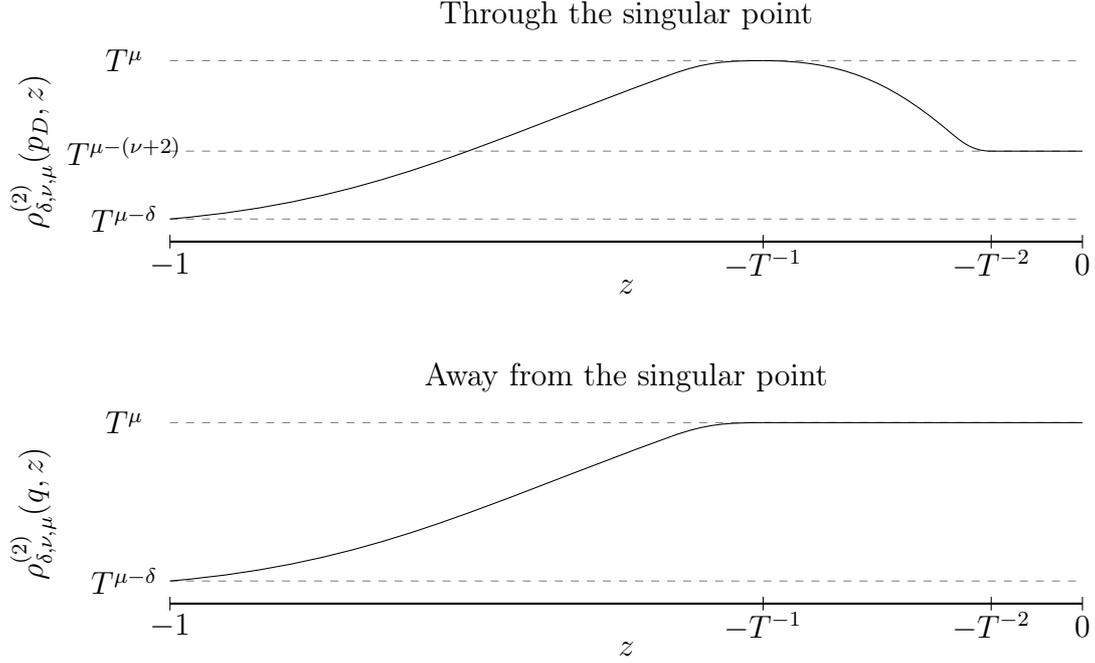 

For carefully chosen parameters, $\rho^{(k+\alpha)}_{\delta,\nu,\mu}$ will give us an appropriate weight function to define weighted H\"older spaces on $(\mathcal{M},g_T)$. As in Section \ref{universalCoversSection}, our definition of these H\"older spaces differs from the standard theory in that distances are measured on local universal covers.

\begin{definition} Fix $T > 0$. Let $K \subset \mathcal{M}$ be a compact subset and $\chi \in T^{r,s}(K)$ an $(r,s)$-tensor field. Let a tilde denote the lift of an object to the Riemannian universal cover of $B_{W(x,T)}(x)$. The weighted $C^{k,\alpha}$ seminorm is defined by 
\[ [ \chi ]_{C^{k,\alpha}_{\delta,\nu,\mu}(x)} = \sup \limits_{\tilde{y} \in B_{W(x,T)}(\tilde{x})} \rho^{(k+\alpha)}_{\delta,\nu,\mu}(x) \frac{\vert \nabla^k\tilde{\chi}(\tilde{x}) - \nabla^k \tilde{\chi}(\tilde{y}) \vert}{\tilde{d}(\tilde{x},\tilde{y})^{\alpha}}.
\]
and 
\[ [ \chi ]_{C^{k,\alpha}_{\delta,\nu,\mu}(K)} = \sup \limits_{x \in K}  \, [ \chi ]_{C^{k,\alpha}_{\delta,\nu,\mu}(\mathcal{M})}(x).\] 
The difference $\vert \nabla^k\tilde{\chi}(\tilde{x}) - \nabla^k \tilde{\chi}(\tilde{y}) \vert$ is computed by parallel transporting $\nabla^k \tilde{\chi}(\tilde{y})$ to $\tilde{x}$. We also write
\[ [ \chi ]_{C^{k}_{\delta,\nu,\mu}(K)} = \sup \limits_{x \in K} \, [ \chi ]_{C^{k}_{\delta,\nu,\mu}(x)} = \sup \limits_{x \in K} \, \vert \rho^{(k)}_{\delta,\nu,\mu} \nabla^k \chi \vert.\] 
The weighted $C^{k,\alpha}$ norm of $\chi$ is then constructed from the seminorm in the usual way:  
\[ \Vert \chi \Vert_{C^{k,\alpha}_{\delta,\nu,\mu}(K)} = \sum_{m=0}^k [ \chi ]_{C^{m}_{\delta,\nu,\mu}(K)} + [ \chi ]_{C^{k,\alpha}_{\delta,\nu,\mu}(K)}.\]
 \end{definition}
 
 We can now state the main theorem of this section.

\begin{theorem}\label{schauderGlobal} For each $T$, define the operator
\[ \mathcal{L}_T = \Delta_{g_T}-1.\]
There exists $\delta_0 > 0$ such that for all $\alpha \in (0,1)$, if $\nu \in (-2,-3/2)$, $\delta \in (0,\delta_0)$, and $\mu \in (\text{max}(\delta,\nu+2), 1)$,  there exists $C > 0$ such that for large enough $T$, 
\[ [ u ]_{C^{2,\alpha}_{\delta,\nu,\mu}(\mathcal{M})} + [  u ]_{C^{2}_{\delta,\nu,\mu}(\mathcal{M})}  + \Vert u \Vert_{C^{0,\alpha}_{\delta,\nu+2,\mu}(\mathcal{M})}  \leq
C \Vert \mathcal{L}_T u \Vert_{C^{0,\alpha}_{\delta,\nu+2,\mu}(\mathcal{M})} \] 
for all $u \in C^{2,\alpha}(\mathcal{M})$ such that $\frac{\partial u}{\partial n} \vert_{\partial \mathcal{M}} = 0$.
\end{theorem} 
The operator $\mathcal{L}_T$ will arise in Section \ref{perturbation} as the linearization of a functional whose zeros are negative K\"ahler-Einstein metrics.

\subsection{Local rescaled geometries}
The key ingredient in the proof of Theorem \ref{schauderGlobal} is the following characterization of the rescaled geometry of the model spaces $(\mathcal{M}, g_T)$.
\begin{proposition}\label{rescaledGeo}
Let $(T_j)_{j=1}^{\infty}$ be a sequence of positive real numbers tending to infinity. If $(x_j)_{j=1}^{\infty}$ is a sequence of points with $x_j \in \mathcal{M}$, then there exists a subsequence (which we also call $j$) such that as $j \rightarrow \infty$, $(\mathcal{M},W(x_j,T_j)^{-2}g_{T_j},x_j)$ converges in the pointed Gromov-Hausdorff topology to one of the following spaces:
\begin{enumerate}
\item the Taub-NUT space $\mathbb{C}^2_{TN}$,
\item the Riemannian product $\mathbb{C} \times \mathbb{R}$,
\item the Riemannian cylinder $D \times \mathbb{R}$,
\item the Calabi model space $(\mathcal{C}_{\pm},g_{\mathcal{C}_{\pm}})$. 
\end{enumerate}
In case $1$, convergence is in the pointed $C^{2,\alpha}$ topology. In cases $2$ and $3$, convergence is in $C^{k,\alpha}$ on local universal covers away from $p$ for $k \geq 0$. In case $4$, convergence is in the pointed $C^{k,\alpha}$ topology for $k \geq 0$. 
\end{proposition}
 
\begin{proof}
Let  $W_j  = W(x_j,T_j)$ and define $r_w$ as in Section \ref{expansionsNearSing}. We will also write $r_w$ for the pullback of $r_w$ by $\pi: \mathcal{M} \rightarrow D \times [-1,1/2]$. In addition, let $\hat{g}_j = W_j^{-2}g_{T_j}$ denote the rescaled metric and $\hat{B}_R(x) = B_{W_jR}(x)$ denote the $R$-ball around a point $x \in \mathcal{M}$ with respect to the rescaled metric. 

We consider four possible behaviors of $(x_j)_{j=1}^\infty$ such that there must exist a subsequence falling into at least one of these categories. \\

\noindent \textbf{Case 1.} $Tr_w(x_j) \rightarrow C < \infty$. Using the coordinates defined in Section \ref{expansionsNearSing} and writing $\omega_{\mathbb{C}} = \frac{i}{2} dy \wedge d\bar{y}$, we have by Equation \ref{chi_expansion} and the discussion in Section \ref{compactification} that 
\begin{dmath}  \omega_T = \pi^{\ast} \chi \omega_D  +  dz \wedge \Theta 
= \pi^{\ast}\omega_D +  \bigg(\frac{1}{2Tr_w} + \frac{1}{T}(g^{2,\alpha} + w^1) \bigg)\pi^{\ast}\omega_{\mathbb{C}}  + \frac{1}{T} dw \wedge \Theta
= \pi^{\ast}\omega_D +  \bigg(\frac{1}{2Tr_w} + \frac{1}{T} (g^{2,\alpha} + w^1 )\bigg)\pi^{\ast}\omega_{\mathbb{C}}  + \frac{1}{T} dw \wedge (\Theta_0 + \theta)
\end{dmath} 
Meanwhile, since the projection $\pi: \mathcal{M} \rightarrow D \times [-1,1/2]$ is modelled on the Hopf fibration, we have by Equation \ref{taubNutCoords} that 
\begin{dmath} 
\frac{1}{T} \omega_{TN,T} = \bigg(\frac{1}{2Tr_w}+1\bigg)\pi^{\ast}\omega_{\mathbb{C}} + \frac{1}{T}  dw \wedge \Theta_0,
\end{dmath}
where $\omega_{TN,T}$ is the Taub-NUT metric in the complex-valued coordinates $u_1$, $u_2$ defined in Section \ref{compactification} and with parameter $T$. Thus
\[ \omega_T - \frac{1}{T}\omega_{TN,T} = \pi^{\ast}(\omega_D - \omega_{\mathbb{C}}) + \frac{1}{T} (g^{2,\alpha} + w^1) \pi^{\ast} \omega_{\mathbb{C}} + \frac{1}{T} dw \wedge \theta.\]

Now change coordinates such that $\underline{u}_i = \frac{1}{\sqrt{T}W} u_i$, $\underline{y} = \frac{1}{TW^2} y$, and $\underline{w} = \frac{1}{TW^2} w$. Then we let $\underline{s}^2 = \frac{1}{TW^2} s^2 = \vert \underline{u}_1 \vert^2+ \vert \underline{u}_2  \vert^2$.  By Equation \ref{rescale_TN}, we have
\[  \omega_{TN,T} = TW^2 \underline{\omega}_{TN,T^2 W^2}.\] 
In addition, noting that $\pi^{\ast} d\underline{y} = \underline{u}_1 d\underline{u}_2 + \underline{u}_2 d\underline{u}_1 = \mathcal{O}(\underline{s})$ with respect to the flat metric on $\mathbb{C}^2_{\underline{u}_1,\underline{u}_2}$, we have that 
\begin{dmath*} 
\pi^{\ast}(\omega_D - \omega_{\mathbb{C}}) = \mathcal{O}(r_w^2) \pi^{\ast}(dy \wedge d\bar{y}) 
= (TW^2)^{4} \mathcal{O}(\underline{s}^6). 
\end{dmath*}
Similarly 
\[  (g^{2,\alpha} + w^1) \pi^{\ast} \omega_{\mathbb{C}} = (TW^2)^2 \underline{g}^{2,\alpha}.\]
Therefore \[ \Vert W^{-2} \omega_T - \underline{\omega}_{TN,T^2W^2}  \Vert_{C^{2,\alpha}} =  \mathcal{O}(T^{-1})\]  in $\mathbb{C}^2_{\underline{u}_1,\underline{u}_2}$ on any region where $\underline{s}$ is bounded. Since $\underline{s}$ is bounded on a ball of $(W_j^{-2}g_{T_j})$-radius $R$ around $x_j$ for any $R > 0$, this gives pointed $C^{2,\alpha}$ convergence. \\

\noindent \textbf{Case 2.} $Tr_w(x_j) \rightarrow \infty$, $r_w(x_j) \rightarrow 0$.
Note that
\begin{equation}\label{g_formula}
 g_T =  \pi^{\ast} \tilde{g} + h_T \, dz^2 + h^{-1}_T \Theta^2,
 \end{equation}
is the Riemannian metric corresponding to $\omega_T$, where $\tilde{g}$ is the Riemannian metric corresponding to $\tilde{\omega}$. By Equation \ref{Wdef}, $W_j = r_w(x_j)$ for sufficiently large $j$, so we study the regularity of $\hat{g}_j = r_w(x_j)^{-2} g_{T_j}$.  

Integrating $h$ along $w$ shows that for any $q \in \mathcal{M}$,
\begin{equation}\label{arcsinh}
\vert\text{ArcSinh}(w(x_j))-\text{ArcSinh}(w(q)) \vert \leq  d_{g_{T_j}}(x_j,q) + \mathcal{O}(T^{-1/2}).
\end{equation}
Now take $q \in \hat{B}_{R}(x_j) = B_{r_w(x_j)R}(x_j)$. Since $w(x_j) \rightarrow 0$ and $d_{g_{T_j}}(x_j,q) \rightarrow 0$,  $w(q)$ can be assumed arbitrarily small. This lower-bounds the derivative of ArcSinh, allowing us to conclude that \[\vert w(q) - w(x_j) \vert \leq 2 r_w(x_j) R\] for large enough $j$. Thus if for each $j$ we make the change of coordinates $\underline{w} = r_w(x_j)^{-1} w$, $\underline{y} = r_w(x_j)^{-1} y$, and $\underline{u}_i = r_w(x_j)^{-1/2}u_i$, we have that $\underline{w}(x_j) \leq 1$ and $\vert \underline{w}(q) - \underline{w}(x_j) \vert \leq 2R$.

Now assume $R > 1$ and $q \in \hat{B}_R(x_j) \setminus \hat{B}_{\epsilon'}(p)$ for some small $\epsilon' > 0$. Thus $r_w(q) \geq C\epsilon' r_w(x_j)$.
By Equation \ref{deltah_expansion_singularpt},
\begin{align*}
h_T &= \frac{T^2}{r_w(x_j)^2 \underline{w}^2+1} + \frac{T}{2 r_w} + T(g^{2,\alpha}+w^1)
\end{align*} 
Thus
\[ r_w(x_j)^{-2} h_T \, dz^2 = \bigg(\frac{1}{r_w(x_j)^2 \underline{w}^2 + 1} + \mathcal{O}((Tr_w(x_j))^{-1})\bigg) d\underline{w}^2,\]
and so
\[ \vert \pi^{\ast}(r_w(x_j)^{-2} h_T \, dz^2 - d\underline{w}^2) \vert \rightarrow 0\]  
pointwise uniformly on the annulus with respect to $\mathbb{C}^2_{\underline{u}_1,\underline{u}_2}$. 

By analogous arguments, the growth of $h_T$ implies that \[ \vert r_w(x_j)^{-2} h_T^{-1} \Theta^2 \vert \rightarrow 0,\]
with convergence as described in the previous paragraph. To analyze the $\tilde{g}$ term, we observe that by Equation \ref{chi_expansion},
\[ r_w(x_j)^{-2} \pi^{\ast}\tilde{g} = \bigg(1+ \frac{1}{2Tr_w} + T^{-1} \mathcal{O}(1)\bigg) \frac{i}{2} d\underline{y} \wedge d\overline{\underline{y}} = \frac{i}{2} d\underline{y} \wedge d\overline{\underline{y}} +  \mathcal{O}((Tr_w)^{-1}). \] 
In summary,
\[ \hat{g}_T \rightarrow  \frac{i}{2} d\underline{y} \wedge d\underline{\overline{y}} + d\underline{w}^2 = \pi^{\ast}g_{\mathbb{C}_{\underline{y}}\times\mathbb{R}}\]
on $\hat{B}_R(x_j) \setminus \hat{B}_{\epsilon'}(x_j)$. A diagonal argument in $j$ and $\epsilon'$ therefore gives Gromov-Hausdorff convergence $\hat{B}_R(x_j) \rightarrow B_R(x_{\infty}) \subset \mathbb{R}^3$.

Now pass to a local universal cover $\tilde{\hat{B}}_\epsilon(q)$ for $q \in \hat{B}_R(x_j) \setminus \hat{B}_{\epsilon'}(p)$ and assume that $\hat{B_{\epsilon}}(q) \subset (\hat{B}_{\epsilon'}(p))^c$.  If we choose a coordinate such that $\Theta = dt$, then locally
\[\tilde{\hat{B}}_\epsilon(q) \simeq  U \times \mathbb{R}_t\] for some $U \subseteq \mathbb{D} \times [-1,1/2]$. Since $T^{-2}h_T \rightarrow 1$, we can rescale the $t$  coordinate to $\underline{t}$ so that on the local universal cover
\[ \hat{g}_T \rightarrow \pi^{\ast} g_{\mathbb{R}}^3 + d\underline{t}^2\]
in $C^{\infty}$, since all terms are smooth away from $p$. \\

\noindent \textbf{Case 3.} $r_w(x_j) \not \rightarrow 0$, $w(x_j) \not \rightarrow \infty$. Assume that the limit of $r_w$ is such that $W_j = 1$ and $\hat{g}_j = g_{T_j}$ for large enough $j$. Passing to a subsequence, we can assume that $w(x_j) \rightarrow w_{\infty} < \infty$ and $d(x_j,p) \rightarrow d_{\infty} < \infty$, since the diameter of $\tilde{\omega}(z)$ and the $S^1$ fiber are bounded. 

Now fix $R > 1$ and consider the domain $\hat{B}_R(x_j) \setminus \hat{B}_{\epsilon'}(p)$ for some small $\epsilon'$. As above, for any $q \in \hat{B}_R(x_j) \setminus \hat{B}_{\epsilon'}(p)$ we can pass to a local universal cover $\tilde{\hat{B_{\epsilon}}}(q) \simeq U \times \mathbb{R}_t$ with coordinates $y, \, \bar{y}, \, w$ on $U$ and $dt = \Theta$.  We analyze the metric on the universal cover as in the previous bullet, except there is no need to rescale by $W_j$.

Since we may assume $\hat{B_{\epsilon}}(q) \subset (\hat{B}_{\epsilon'}(p))^c$, $r_w$ is bounded below on $\tilde{\hat{B_{\epsilon}}}(q)$. Thus we have by Equation \ref{chi_expansion} and Proposition \ref{deltah_decay_prop} that 
\[  \pi^{\ast}\tilde{g} \rightarrow \pi^{\ast} g_D.\]
in $C^{2,\alpha}$ in these coordinates. Similarly, by the discussion in Section \ref{hcorrection},
\[  h_T  \rightarrow T^2\bigg(\frac{1}{w^2+1}+ \mathcal{O}(T^{-1})\bigg). \]
Thus by a rescaling of $t$, 
\[ \hat{g}_T \rightarrow \pi^{\ast} g_D + \frac{1}{w^2+1} dw^2 + (w^2+1) \, dt^2\] 
in $C^{\infty}$ on the local universal cover. Meanwhile, since $h_T^{-1} \rightarrow 0$, this gives Gromov-Hausdroff convergence to $D \times \mathbb{R}$. \\

\noindent \textbf{Case 4.} $w(x_j) \rightarrow \infty$. In this case we must have that $z(x_j) \rightarrow z_{\infty} \in [-1,1/2]$, and again we have $\hat{g}_j = g_{T_j}$. Fix $R > 0$ and take $q \in \hat{B}_R(x_j)$. Then
\begin{equation}\label{w_far}
d(q,x_j) \sim \vert \text{ArcSinh}(w(q)) -  \text{ArcSinh}(w(x_j)) \vert \sim \vert \log(w(q))-\log(w(x_j)) \vert,
\end{equation}
so 
\[ \inf \limits_{q \in \hat{B}_R(x_j)} r_w(q) \rightarrow \infty\]
as $j \rightarrow \infty$. 

First assume that $z_\infty = 0$ and define $\underline{w} = \frac{w(q)}{w(x_j)}$. Then by the discussion in Section \ref{hcorrection},
\begin{equation}\label{h_expression_6}
h(q)\, dz^2 = \bigg(\frac{k_{\pm}z(x_j)\underline{w}(q) + 1}{\frac{2}{3}k_{\pm}(z(x_j)\underline{w}(q))^3 + (z(x_j)\underline{w}(q))^2 + T^{-2}}\bigg)   \, z(x_j)^2 \,  d\underline{w}^2 \rightarrow \frac{1}{\underline{w}^2} \, d\underline{w}^2.
\end{equation}
By arguments similar to the previos section, the metric collapses to $D \times \mathbb{R}$, while
\[ \hat{g}_T \rightarrow \pi^{\ast} g_D + \frac{1}{\underline{w}^2} d\underline{w}^2 + \underline{w}^2 \, dt^2\]
in $C^{\infty}$ on the local universal cover. 

Now assume $z_\infty > 0$. Take $q \in \hat{B}_R(x_j)$ with coordinates $y, \, \bar{y}, \, z$ and $t$ as above.  As previously the discussion in Section \ref{hcorrection} gives that for $z < 0$,
\begin{equation}\label{h_expression_6}
h = \frac{k_- z+1}{\frac{2}{3}k_-  z^3 + z^2 + T^{-2}} + \mathcal{O}(T^{-5}z^{-4}).
\end{equation}
Meanwhile, Remark \ref{deltachi_D} gives that \[\delta \chi= \mathcal{O}(T^{-1}), \ \  \vert w \vert = C_2.\] Using Equations \ref{maineqn2} and \ref{deltah_decay_laplac} to integrate $\delta \chi$ outside of this region, we find that
\begin{equation}\label{chi_expression_6} \chi = 1 + k_{\pm} z + \mathcal{O}(T^{-1}).
\end{equation} 
Thus
\[ \hat{g}_T \rightarrow (1+k_- z)\pi^{\ast} g_D + \frac{k_{-}z + 1}{\frac{2}{3}k_-z^3+z^2} \, dz^2 + \frac{\frac{2}{3}k_-z^3+z^2}{k_{-}z + 1} \Theta^2,\]
in $C^{\infty}$, and this metric is recognized to be the Calabi model space metric described in Section \ref{calabimodel} for $n = k_-$. The discussion is similar for $z > 0$.

Note that the distance from $x_j$ to $\partial M$ is finite if $\vert z_{\infty} \vert > 0$, since
\begin{equation}\label{disttobound} \lim \limits_{j \rightarrow \infty} (\text{ArcSinh}(T_j)-\text{ArcSinh}(w(x_j))) = - \lim \limits_{j \rightarrow \infty} \log(T_j^{-1} w(x_j)).
\end{equation}
Thus in this case a sufficiently large definite ball around $x_{\infty}$ will contain a portion of the boundary. 

%

\end{proof} 
%

\subsection{Schauder estimates on $(\mathcal{M},\omega_T)$} The usefulness of Proposition \ref{rescaledGeo} will become apparent in the proofs of Propositions \ref{schauderLocal} and \ref{schauderGlobal}. In the first proof, we will see that the proposition gives a lower bound on the regularity scale of $(\mathcal{M},\omega_T)$, directly implying a local Schauder estimate. In the second proof, we will assume that a global Schauder estimate fails to hold and take a sequence of functions violating the estimate. After passing to a rescaled limit, the sequence will converge to a function whose behavior contradicts the properties of the spaces described in Proposition \ref{rescaledGeo}.

Proposition \ref{schauderLocal} is modeled on \cite{sz19} Proposition 4.37. 

\begin{proposition}\label{schauderLocal} (Local weighted Schauder estimate). For $T \gg 0$ and $(\mathcal{M}, g_T)$ as constructed in Section \ref{m_construction}, the following estimates hold:
\begin{enumerate}
\item (Interior estimate) For all $\alpha \in (0,1)$, there exists $C_{\alpha} > 0$ such that for all $x \in \mathcal{M}^{o}$, $r \in (0,\frac{1}{8}]$, and $u \in C^{2,\alpha}(B_{2rW(x)}(x))$,
\[ r^{2+\alpha} \Vert u \Vert_{C^{2,\alpha}_{\delta,\nu,\mu}(B_{rW(x)}(x))} \leq C_{\alpha} ( \Vert \mathcal{L}_T u \Vert_{C^{0,\alpha}_{\delta,\nu+2,\mu}(B_{2rW(x)}(x))} +  \Vert u \Vert_{C^{0}_{\delta,\nu,\mu}(B_{2rW(x)}(x))}).\] 
\item (Boundary estimate) For all $\alpha \in (0,1)$, there exists $C_{\alpha} > 0$ such that for all $x \in \partial \mathcal{M}$, $r \in (0,\frac{1}{8}]$, and $u \in C^{2,\alpha}(B_{2rW(x)}(x))$,
\begin{align*}
r^{2+\alpha} \Vert u \Vert_{C^{2,\alpha}_{\delta,\nu,\mu}(B_{rW(x)}(x))} \leq C_{\alpha} \bigg( &\Vert \mathcal{L}_T u  \Vert_{C^{0,\alpha}_{\delta,\nu+2,\mu}(B_{2rW(x)}(x))} \\& +  \Vert u \Vert_{C^{0}_{\delta,\nu,\mu}(B_{2rW(x)}(x))}  + \bigg\Vert \frac{\partial u}{\partial n} \bigg\Vert_{C^{1,\alpha}_{\delta,\nu+1,\mu}(B_{2rW(x)}(x))} \bigg).
\end{align*} 
\end{enumerate} 
\end{proposition}
\begin{proof}
Fix some small $\epsilon' > 0$. We observe that Proposition \ref{rescaledGeo} demonstrates that the $C^{2,\alpha},\epsilon'$-regularity scale $r_{2,\alpha}(x)$ of $(\mathcal{M},g_T)$ with respect to $\epsilon'$ at $x$ is bounded below by a multiple of $W(x,T_j)$:
\begin{equation}\label{regscale_lb}
r_{2,\alpha}(x) \geq CW(x,T_j).
\end{equation} 
For arguing by contradiction, this is equivalent to saying that for any sequence $(x_j,T_j)_{j=1}^\infty$, the $C^{2,\alpha}$-regularity scale of $(\mathcal{M},W(x_j,T_j)^{-2}g_T)$ is bounded below by a constant, not necessarily uniform across sequences, at $x_j$. This follows because $(\mathcal{M},W(x_j,T_j)^{-2}g_T,x_j)$ converges to one of the four spaces in the proposition in $C^{2,\alpha}$ on the local universal cover.

Equation \ref{regscale_lb} allows us to apply the local Schauder estimates on Euclidean balls with respect to the metric $\hat{g}_T$. On any such ball $B=\hat{B}_r(x)$ with $2B = \hat{B}_{2r}(x)$,
\[ \Vert u \Vert_{C^{2,\alpha}(B)} \leq C(\Vert (\Delta_{\hat{g}}-W^2) u \Vert_{C^{0,\alpha}(2B)} + \Vert  u \Vert_{C^{0}(2B)})\]
for $C$ uniform in $T$ because $W(x)$ is uniformly bounded above by $1$. Now we will demonstrate in the proof of Theorem \ref{schauderGlobal} that for any fixed $r$, $\rho^{(0)}_{\delta,\nu,\mu}$ is equivalent to a constant on $B_{rW(x)}(x) = \hat{B}_r(x)$ uniformly in $x \in \mathcal{M}$ and $T \gg 0$. Therefore, multiplying through by $\rho^{(0)}_{\delta,\nu,\mu}(x)$ and rescaling $\hat{g}_T \rightarrow g_T$ yields that 
\begin{equation}\label{almostschauder} \Vert u \Vert_{C^{2,\alpha}_{\delta,\nu,\mu}(B)} \leq C( \Vert \mathcal{L}_T u \Vert_{C^{0,\alpha}_{\delta,\nu+2,\mu}(2B)} + \Vert  u \Vert_{C^{0}_{\delta,\nu,\mu}(2B)}).
\end{equation}
The proof of the boundary estimate is similar.

\end{proof}
A global version of Proposition \ref{schauderLocal} follows from the local version and a covering argument.  
\begin{proposition}\label{schauderGlobalC0} Let $\nu$, $\alpha$, and $\delta$ be parameters satisfying the conditions described in Theorem \ref{schauderGlobal}. Then there exists $C > 0$ such that for all sufficiently large $T$ and $u \in C^{2,\alpha}_{\delta,\nu,\mu}(\mathcal{M})$, 
\[ \Vert u \Vert_{C^{2,\alpha}_{\delta,\nu,\mu}(\mathcal{M})} \leq C \bigg( \Vert \mathcal{L}_T u \Vert_{C^{0,\alpha}_{\delta,\nu+2,\mu}(\mathcal{M})} + \bigg\Vert \frac{\partial u}{\partial n} \bigg\Vert_{C^{1,\alpha}_{\delta,\nu+1,\mu}(\partial \mathcal{M})} + \Vert u \Vert_{C^{0}_{\delta,\nu,\mu}(\mathcal{M})}\bigg).\] 
\end{proposition}

We can now give the proof of Theorem \ref{schauderGlobal}.

\begin{proof} (Theorem \ref{schauderGlobal}) If there is no such $C$, then there must be a sequence of $T_j$ and $C^{2,\alpha}$ functions $u_j$ with $\frac{\partial u_j}{\partial n}\vert_{\partial \mathcal{M}} = 0$ such that 
\[ [  u_j ]_{C^{2,\alpha}_{\delta,\nu,\mu}(\mathcal{M})} + [  u_j ]_{C^{2}_{\delta,\nu,\mu}(\mathcal{M})} +  \Vert u_j \Vert_{C^{0,\alpha}_{\delta,\nu+2,\mu}(\mathcal{M})} = 1 \]
while 
\[\Vert \mathcal{L}_{T_j} u_j \Vert_{C^{0,\alpha}_{\delta,\nu+2,\mu}(\mathcal{M})} \rightarrow 0.\] 
Then there must exist a sequence of points $(x_j)$ such that
\begin{equation}\label{rhs_big} [  u_j ]_{C^{2,\alpha}_{\delta,\nu,\mu}(x_j)} + [  u_j ]_{C^{2}_{\delta,\nu,\mu}(x_j)} + [  u_j ]_{C^{0,\alpha}_{\delta,\nu+2,\mu}(x_j)} +  [  u_j ]_{C^{0}_{\delta,\nu+2,\mu}(x_j)} > \frac{1}{8}.
\end{equation}

We will prove a contradiction by passing to the rescaled pointed Gromov-Hausdorff limit for the sequence $(x_j)_{j=1}^\infty$ and the rescaling factor $W(x_j,T)$. To preserve norms after rescaling the metric, we rescale both the functions $(u_j)$ and the weight functions $\rho^{(k+\alpha)}_{\delta,\nu,\mu}$. Recall that if $\hat{g} = \lambda^2 g$ for any metric $g$ and $\lambda > 0$, then formally,
\[ \Vert \lambda^{k+\alpha} u \Vert_{C^{k,\alpha}_{\hat{g}}} =  \Vert u \Vert_{C^{k,\alpha}_g} \]
on functions, and a similar equality holds for higher degree forms. In our case, since we are considering weighted H\"older norms, we have a choice of how to divide the scaling factor between the weight function $\rho^{(k+\alpha)}_{\delta,\nu,\mu}$ and the target function $u_j$, but we must ensure that the rescaled weight functions converge on the appropriate rescaled spaces.  Thus let
\begin{equation}\label{rho_rescaled}
\hat{\rho}^{k+\alpha}_{\delta,\nu,\mu}(q) = W(x_j)^{-(\nu + k+\alpha)} T_j^{-\mu} \rho^{k+\alpha}_{\delta,\nu,\mu}(q) = (1+w(q))^{-\delta}\bigg(\frac{W(q)}{W(x_j)}\bigg)^{\nu+k+\alpha}
\end{equation}
and
\begin{equation}\label{u_rescaled}
\hat{u}_j(q) =   W(x_j)^{\nu} T_j^{\mu} u_j(q).
\end{equation}
In what follows, whenever we take a H\"older norm of $\hat{u}_j$ or a related function, we take the weight to be $\hat{\rho}^{k+\alpha}_{\delta,\nu,\mu}$ and the metric to be $\hat{g}_{j} = W(x_j,T_j)^{-2} g_{T_j}$. Thus 
\begin{equation}\label{rhsu} [ \mathcal{L}_{T_j} \hat{u}_j ]_{C^{0,\alpha}_{\delta,\nu+2,\mu}(\mathcal{M})} \rightarrow 0
\end{equation}
and 
\begin{equation}\label{lhsu}
 [ \hat{u}_j ]_{C^{2,\alpha}_{\delta,\nu,\mu}(\mathcal{M})} + [  \hat{u}_j ]_{C^{2}_{\delta,\nu,\mu}(\mathcal{M})}  + \Vert \hat{u}_j \Vert_{C^{0,\alpha}_{\delta,\nu+2,\mu}(\mathcal{M})}  = 1 
\end{equation}
while
\begin{equation}\label{1_lhs_rescaled}
 \left[ \hat{u}_j \right]_{C^{2,\alpha}_{\delta,\nu,\mu}(x_j)} +  \left[  \hat{u}_j \right]_{C^{2}_{\delta,\nu,\mu}(x_j)} + \left[ \hat{u}_j \right]_{C^{0,\alpha}_{\delta,\nu+2,\mu}(x_j)} +  \left[  \hat{u}_j \right]_{C^{0}_{\delta,\nu+2,\mu}(x_j)}  > \frac{1}{8}.
\end{equation}
As in the proof of Proposition \ref{rescaledGeo}, there must exist a subsequence of $(x_j)_{j=1}^\infty$ satisfying one of four behaviors. \\

\noindent \textbf{Case 1.}  $Tr_w(x_j) \rightarrow C < \infty$. By Proposition \ref{rescaledGeo}, we can view $(\hat{u}_j)_{j=1}^{\infty}$ as a sequence of functions on a neighborhood of a point $x_{\infty}$ in $(\mathbb{C}^2,g_{TN,a})$ for some $a > 1$.
Now by Equation \ref{rho_rescaled},
 \[\hat{\rho}^{(k+\alpha)}_{\delta,\nu,\mu}(q) =  (1 + \mathcal{O}(T^{-1}))\bigg(\frac{W(q)}{W(x_j)}\bigg)^{\nu+k+\alpha}.\]
But in the Taub-NUT region, $\frac{r_w(q)}{W(x_j)}$ converges to $d_{g_{TN,a}}^2(q,0)$ (see Equation \ref{s_vs_r}), so $\hat{\rho}^{(k+\alpha)}_{\delta,\nu,\mu}(q)$ is uniformly equivalent to a constant on $\hat{B}_R(x_j)$ for any fixed $R$ and large enough $j$. Thus by Equation \ref{lhsu}, the unweighted local Schauder estimates discussed in the proof of Proposition \ref{schauderLocal},  and the $C^{2,\alpha}$ convergence of the metrics,
\begin{equation}\label{case1_c2a}
\Vert \hat{u}_j\Vert_{C^{2,\alpha}(B_{R}(x_\infty))}  \leq C(R),
\end{equation}
that is, $\hat{u}_j$ is bounded in the unweighted $C^{2,\alpha}$ toplogy on $B_{R}(x_\infty) \subset (\mathbb{C}^2,g_{TN,a})$ (defined e.g. using Definition \ref{spaces_def} with weight function $\rho^{(k+\alpha)}_{\delta,\nu,\mu} \equiv 1$) with bound depending on $R$, while
\begin{equation}\label{vjdecay}
\hat{u}_j(q) \leq C d_{\hat{g}_{j}}(q,p)^{-2(\nu+2)}. 
\end{equation}
for $C$ independent of $R$. Also, viewing $x_j$ as a point in $B_{R}(x_\infty)$ we have by Equation \ref{1_lhs_rescaled} that 
\begin{equation}\label{vjatxj} 
 \left[  \hat{u}_j \right]_{C^{2,\alpha}(x_j)} +  \left[  \hat{u}_j \right]_{C^{2}(x_j)} +   \left[  \hat{u}_j \right]_{C^{0,\alpha}(x_j)} +  \left[  \hat{u}_j \right]_{C^{0}(x_j)} > \frac{1}{C}
\end{equation} 
for $C$ again independent of $R$. Equation \ref{case1_c2a} gives that for any $\beta < \alpha$ a subsequence converges in $C^{2,\beta}(B_R(x_{\infty}))$ to a function $\hat{u}_{\infty}$, and by Equation \ref{vjdecay}, 
\begin{equation}\label{case1_u_decay}
\hat{u}_{\infty}(q) \leq Cd_{g_{TN,a}}(q,0)^{-2(\nu+2)}
\end{equation}
for $C$ independent of $R$. 
Because we are rescaling the metric, $\mathcal{L}_T$ does not converge to $\Delta_{g_{TN}}-1$. Rather,
\begin{align*}
\Vert \mathcal{L}_T \hat{u}_j \Vert_{C^{0,\alpha}_{\delta,\nu+2,\mu}(\hat{B}_{R}(x_j))} &\geq \Vert W(x_j)^{-2} \Delta_{\hat{g}} \hat{u}_j \Vert_{C^{0,\alpha}_{\delta,\nu+2,\mu}(\hat{B}_{R}(x_j))} - \Vert \hat{u}_j \Vert_{C^{0,\alpha}_{\delta,\nu+2,\mu}(\hat{B}_{R}(x_j))} \\
&\geq CT^2 \Vert  \Delta_{\hat{g}} \hat{u}_j \Vert_{C^{0,\alpha}_{\delta,\nu+2,\mu}(\hat{B}_{R}(x_j))} + \mathcal{O}(1).
\end{align*}
Thus
\[ \Vert \Delta_{\hat{g}} \hat{u}_j \Vert_{C^{0,\alpha}(\hat{B}_{R}(x_\infty))} = \mathcal{O}(T^{-2}).\]
By $C^{2,\beta}$ convergence, this implies that $\hat{u}_{\infty}$ is a $C^2$ harmonic function on $B_R(x_{\infty})$.   Also, standard Schauder estimates on $\hat{B}_{R}(x_\infty)$ now imply that $\hat{u}_j$ converges in $C^{2,\alpha}$. 

Since $R$ was arbitrary, choosing a sequence of $R \rightarrow \infty$ gives a harmonic function on $(\mathbb{C}^2,g_{TN})$. Since $\nu + 2 > 0$, Equation \ref{case1_u_decay} implies that $\hat{u}_{\infty}$ decays at infinity, so by the Cheng-Yau gradient estimate, $\hat{u}_\infty$ must vanish (\cite{Li} Proposition 6.6). But this contradicts Equation \ref{vjatxj} by $C^{2,\alpha}$ convergence. \\


\noindent \textbf{Case 2.}  $Tr_w(x_j) \rightarrow \infty$, $r_w(x_j) \rightarrow 0$.  By similar considerations as in the previous case, we have that $\hat{\rho}^{(k+\alpha)}_{\delta,\nu,\mu}(q)$ is uniformly constant on $\hat{B}_R(x_j) \setminus \hat{B}_{R^{-1}}(p)$. Now by Proposition \ref{rescaledGeo}, $(\mathcal{M},\hat{g}_{j},x_j)$ converges to $(\mathbb{R}^3,g_{\mathbb{R}^3},x_{\infty})$ in the pointed Gromov-Hausdorff topology. It is not immediate that $\hat{u}_j$ induces a function on $\mathbb{R}^3$. Instead, we observe that for each $j$,
\begin{equation}\label{case2_c2a}
\Vert \hat{u}_j \Vert_{C^{2,\alpha}(\hat{B}_{R}(x_j) \setminus \hat{B}_{R^{-1}}(p))} \leq C(R),
\end{equation}
\begin{equation}\label{case2_pointwise}
 \left[  \hat{u}_j \right]_{C^{2,\alpha}(x_j)} +  \left[  \hat{u}_j \right]_{C^{2}(x_j)} +   \left[  \hat{u}_j \right]_{C^{0,\alpha}(x_j)} +  \left[  \hat{u}_j \right]_{C^{0}(x_j)} > \frac{1}{C},
\end{equation}
\begin{equation}\label{case2_decay}
\vert \hat{u}_j(q) \vert \leq C\bigg(\frac{r_w(x_j)}{r_w(q)}\bigg)^{\nu+2},
\end{equation}
and 
\begin{equation}\label{case2_laplace}
\Vert \Delta_{\hat{g}} \hat{u}_j \Vert_{C^{0,\alpha}(\hat{B}_{R}(x_j))} \rightarrow 0.
\end{equation}

Take $q \in B_R(x_\infty) \setminus \lbrace 0 \rbrace$ and assume that $q_j \in (\mathcal{M},g_{T_j}) \rightarrow q$. Proposition \ref{rescaledGeo} implies that there is $\epsilon > 0$ such that $(\tilde{\hat{B_\epsilon}}(q_j),\hat{g}_{j}) \rightarrow (B_R(q) \times \mathbb{R},g_{\mathbb{R}^4})$ in $C^{2,\alpha}$. Pulling back to these covers, Equation \ref{case2_c2a} gives that $\hat{u}_j$ defines a sequence of functions converging to some $\hat{u}_{\infty}$ in $C^{2,\beta}(B_R(q) \times \mathbb{R})$ for any $\beta < \alpha$.  In addition, Equation \ref{case2_laplace} implies that $\hat{u}_{\infty}$ is $\Delta_{\mathbb{R}^4}$-harmonic.

Now since the $S^1$ fiber diameter vanishes as $j \rightarrow \infty$, Equation \ref{case2_c2a} implies that $\hat{u}_{\infty}$ is constant in the $t$ direction. Therefore $\hat{u}_{\infty}$ descends to a harmonic function on $\hat{B}_\epsilon(q)$. Repeating this near each point in $B_R(x_\infty)$ except the origin and then letting $R \rightarrow \infty$, we can take $\hat{u}_{\infty}$ to be defined on $\mathbb{R}^3 \setminus \lbrace 0 \rbrace$.  Finally, Equation \ref{case2_decay} gives that \[ \vert \hat{u}_{\infty}(q) \vert \leq Cd_{\mathbb{R}^3}(q,0)^{-2(\nu+2)}.\]  Now $\nu + 2 \in (0,1/2)$, so a harmonic function on $\mathbb{R}^3 \setminus \lbrace 0 \rbrace$ which is bounded by $d_{\mathbb{R}^3}(q,0)^{-2(\nu+2)}$ must vanish (\cite{sz19} Lemma 5.7). As in the previous case, this implies that $\hat{u}_{\infty} = 0$, contradicting Equation \ref{case2_pointwise}.\\

\noindent \textbf{Case 3.} $r_w(x_j) \not \rightarrow 0$, $w(x_j) \not \rightarrow \infty$.  We can assume that $W(x_j) = 1$ for large $j$. Since in this case $d_{g_{T_j}}(x_j,p) \rightarrow d_{\infty} < \infty$ and $\hat{g}_{j} = g_{T_j}$, we have that  $\hat{\rho}^{(\alpha)}_{\delta,\nu+2,\mu}(q)$ is uniformly constant for $q \in \hat{B}_{R}(x_j) \setminus \hat{B}_{R^{-1}}(p)$. 
By similar arguments as in the previous case, we have that
\begin{equation}\label{case3_c2a}
\Vert \hat{u}_j \Vert_{C^{2,\alpha}(\hat{B}_{R}(x_j) \setminus \hat{B}_{R^{-1}}(p))} \leq C(R),
\end{equation}
 while
\begin{equation}\label{case3_pointwise}
[  \hat{u}_j ]_{C^{2,\alpha}(x_j)} + [  \hat{u}_j ]_{C^{2}(x_j)} + [ \hat{u}_j ]_{C^{0,\alpha}(x_j)} +  [  \hat{u}_j]_{C^{0}(x_j)} > \frac{1}{C}
\end{equation}
and
\begin{equation}\label{case3_decay}
\vert \hat{u}_j(q) \vert \leq C(1+w(q))^{\delta}r_w(q)^{-(\nu+2)},
\end{equation}
for $C$ independent of $R$, and 
\begin{equation}\label{case3_laplace}
\Vert \mathcal{L}_T \hat{u}_j \Vert_{C^{0,\alpha}(\hat{B}_{R}(x_j) \setminus \hat{B}_{R^{-1}}(p))} \rightarrow 0.
\end{equation}
As in the previous bullet, $\hat{u}_j$ converges to an $S^1$ invariant function on each local universal cover, and these functions stitch to a function $\hat{u}_{\infty}$ on the punctured cylinder $D \times \mathbb{R} \setminus \lbrace p \rbrace$ such that $\vert \hat{u}_{\infty}(q) \vert \leq C(1+w(q))^{\delta}$ for $q$ away from $p$. Further, Equation \ref{case3_decay} and Equation \ref{case3_laplace} give that $\hat{u}_{\infty}$ is a weak solution to $\mathcal{L}_T u = 0$ on $D \times \mathbb{R}$.  It is computed that 
\[ \Delta_{\hat{g}_{j}} = \pi^{\ast} \Delta_{g_D} +(1+w^2)\,  \partial_w^2  + 2w \, \partial_w + \frac{1}{1+w^2} \partial_{t}^2   + \mathcal{O}(T^{-1})\]
in coordinates on each local universal cover (see the proof of Proposition \ref{rescaledGeo}), so $u_{\infty}$  
satisfies
\[ \pi^{\ast} \Delta_{D} \hat{u}_\infty + (1+w^2) \, \partial_w^2 \hat{u}_{\infty}  + 2w \, \partial_w \hat{u}_{\infty} - \hat{u}_{\infty} = 0 \]
on $D \times \mathbb{R}$.

As above, we characterize this equation by separation of variables. Let $x = \frac{1}{2}(1+iw)$. The resulting ODE is written
\[ x(x-1) f_{xx} + (2x-1) f_x - (1+\lambda^2)f = 0.\] 
This is a hypergeometric equation with 
\[ \alpha(\lambda) = \frac{1+\sqrt{5+4\lambda^2}}{2} \ \ \ \ \ \ \ \beta(\lambda) = \frac{1-\sqrt{5+4\lambda^2}}{2} \ \ \ \ \ \ \ \gamma = 1\]
in the notation of Section \ref{solution}. As previously, the parameters satisfy the relation
\[ \frac{\alpha(\lambda) + \beta(\lambda) + 1}{2} = \gamma.\] 
Therefore the arguments from Section \ref{solution} give that the two fundamental solutions $F(\alpha,\beta,\gamma;x(z))$ and $F(\alpha,\beta,\gamma;1-x(z))$, each of which grows like $\vert w \vert^{-\beta}$ for large $\vert w \vert$. (Recall that we needed the inhomogeneous delta function term to achieve globally decaying solutions in Section \ref{solution}.) Since $-\beta > \frac{\sqrt{5}-1}{2}$, choosing $\delta  < \delta_0 < \frac{\sqrt{5}-1}{2}$ gives that $u_{\infty} = 0$, a contradiction. \-\ \\

\noindent \textbf{Case 4.}  $w(x_j) \rightarrow \infty$. Recall that $d_{g_{T_j}}(x_j,\delta M) \rightarrow \infty$ if and only if $z_{\infty} = 0$. First assume that this is the case. Now $d_{g_{T_j}}(x_j,p) \rightarrow \infty$, so $\hat{B}_R(x_j)$ converges to a ball in the Calabi model space in the rescaled coordinate $\underline{w}(q) = \frac{w(q)}{w(x_j)}$.  We have by Equation \ref{w_far} that \[ Ce^{-R} \leq \vert \underline{w}(q) \vert \leq Ce^R.\]

Now \[\hat{\rho}^{(k+\alpha)}_{\delta,\nu,\mu}(q) = (1+w(q))^{-\delta}\] on $\hat{B}_R(x_j)$ for any $R > 0$ and large enough $j$. Thus if we define $\hat{v}_j = (1+w(x_j))^{-\delta} \hat{u}_j$, then we have that
\begin{equation}\label{case3_c2a}
\Vert \hat{v}_j \Vert_{C^{2,\alpha}(\hat{B}_{R}(x_j))} \leq C(R),
\end{equation}
\begin{equation}\label{case4_xj}
 [\hat{v}_j ]_{C^{2,\alpha}(x_j)} + [  \hat{v}_j ]_{C^{2}(x_j)} + [ \hat{v}_j ]_{C^{0,\alpha}(x_j)} +  [  \hat{v}_j]_{C^{0}(x_j)} > \frac{1}{C}
\end{equation}
\begin{equation}\label{case4_growth}
\vert \hat{v}_j(q) \vert \leq C\frac{(1+w(q))^{\delta}}{(1+w(x_j))^{\delta}} \leq C\underline{w}^{\delta}(q),
\end{equation}
and
\begin{equation}\label{case4_lt}
\vert \mathcal{L}_{T_j}\hat{v}_j(y) \vert \rightarrow 0.
\end{equation}
As above, we get an $S^1$-invariant limit function $\hat{v}_{\infty}$.  Equation \ref{case4_growth} implies that $\hat{v}_{\infty} = \mathcal{O}(\underline{w}^{\delta})$. Using the expression we have derived for the metric in these coordinates and arguing as in the previous case, we have that 
\begin{equation}\label{case4_pde} \pi^{\ast}\Delta_D\hat{v}_{\infty} + \underline{w}^2 \partial_{\underline{w}}^2\hat{v}_{\infty} + 2\underline{w}\, \partial_{\underline{w}}\hat{v}_{\infty} -\hat{v}_{\infty}  = 0.
\end{equation}
Using separation of variables, the resulting ODE is 
\[ x^2 f_{xx} + 2x f_x -(1+\lambda^2) f = 0,\]
whose solutions are given by power functions in $\underline{w}$ with exponents contained in a discrete set. The decaying solutions blow up near $\underline{w}(q) = 0$, so a global solution to the ODE must be growing like some definite power of $\underline{w}$. This contradicts Equation \ref{case4_growth} for small enough $\delta$. 

Now allow $ z_{\infty} > 0$, so for large enough $R$, $\hat{B}_R(x_j)$ contains a portion of the boundary. We have $\hat{v}_{\infty}$ as previously, but we cannot consider the behavior as $\underline{w}(q) \rightarrow \infty$ because $\underline{w}(q)$ is bounded. Instead, we argue that by Equation \ref{case4_pde}, the Hopf maximum principle, and the Neumann boundary condition on $u_j$, $\hat{v}_{\infty}$ cannot achieve its maximum on $\delta \mathcal{M}$. On the other hand, by Equation \ref{case4_growth}, \[ \hat{v}_{\infty}(q) \leq  C\left(\frac{z(q)}{z_\infty}\right)^{\delta},\] so since $\delta > 0$, $\hat{v}_{\infty}(q)$ decays to zero away from the boundary of $\mathcal{C}$. Thus $\hat{v}_{\infty} \equiv 0$, contradicting $C^{2,\alpha}$ convergence and Equation \ref{case4_xj}.

\end{proof}

\section{Perturbation to an Exact Solution}\label{perturbation}

\subsection{Implicit function theorem} Our last step is to apply the implicit function theorem in the following form.

\begin{theorem}\label{implicitfxnthm} Let $\mathcal{F}: S_1 \rightarrow S_2$ be a map between Banach spaces such that for all $v \in S_1$, 
\[ \mathcal{F}(v) - \mathcal{F}(0) = \mathcal{L}(v) + \mathcal{N}(v) \]
for operators $\mathcal{L}$ and $\mathcal{N}$ with the following properties:
\begin{enumerate}
\item (Bounded inverse) $\mathcal{L}$ is a linear isomorphism and there exists $C_L > 0$ such that
\[ \Vert \mathcal{L}^{-1} \Vert_{\text{op}} \leq C_L.\]
\item (Controlled nonlinear error) We have that $\mathcal{N}(0) = 0$, and there exists $C_N > 0$ and $r_0 \in (0, \frac{1}{2 C_N C_L})$ such that for all $v_1, v_2 \in \overline{B_{r_0}(0)} \subset S_1$, 
\[ \Vert \mathcal{N}(v_1) - \mathcal{N}(v_2) \Vert_{S_2} \leq C_N r_0 \Vert v_1 - v_2 \Vert_{S_1}.\]
\item (Controlled initial error) The radius $r_0$ can be chosen such that \[\Vert \mathcal{F}(0) \Vert_{S_2} \leq \frac{r_0}{4C_L}.\]  
\end{enumerate} 
Then there exists a unique $x \in B_{r_0}(0)$ such that \[ \mathcal{F}(x) = 0.\] Further,
\begin{equation}\label{solutionbound}
\Vert x \Vert_{S_1} \leq 2C_L \Vert \mathcal{F}(0) \Vert_{S_2}.
\end{equation} 
\end{theorem} 
In our case
\[ S_1 = \left\lbrace \phi \in C^{2,\alpha}(\mathcal{M}), \ \phi \text{ is $S^1$ invariant and } \frac{\partial \phi}{\partial n} \bigg\vert_{\partial M} = 0 \right\rbrace\]
with the norm 
\begin{equation}\label{s1norm}
 \Vert \phi \Vert_{S_1} =   [ \phi  ]_{C^{2,\alpha}_{\delta,\nu,\mu}(\mathcal{M})} + [   \phi ]_{C^{2}_{\delta,\nu,\mu}(\mathcal{M})} + \Vert  \phi  \Vert_{C^{0,\alpha}_{\delta,\nu+2,\mu}(\mathcal{M})}
 \end{equation}
with $\delta$, $\nu$, and $\mu$ as in Theorem \ref{schauderGlobal} and
\[ S_2 = \lbrace f \in C^{0,\alpha}_{\delta,\nu+2,\mu}(\mathcal{M}): \, f \text{ is $S^1$-invariant} \rbrace\]
with the norm
\begin{equation}\label{s2norm} 
\Vert f \Vert_{S_2} = \Vert f \Vert_{C^{0,\alpha}_{\delta,\nu+2,\mu}(\mathcal{M})}.
\end{equation}
Note that these spaces depend on $T$.

For each $T$ we seek a function $\mathcal{F}_T: S_1 \rightarrow S_2$ with the property that if $\mathcal{F}_T(v) = 0$ then $\omega_v = \omega_T + i \partial \bar{\partial} v$ is K\"ahler-Einstein. Recall that in our notation
\[ \text{Ric} \, \omega_T = -i \partial \bar{\partial} \log\bigg( \frac{\chi}{h} \bigg) - \pi^{\ast} \omega_D\]
where $\omega_T$ is the solution constructed in Section \ref{solution} for a fixed $T$ and $\chi$ and $h$ implicitly depend on $T$.  Now say there exists a K\"ahler potential $\phi$ such that 
\[ \omega_T = \pi^{\ast} \omega_D + i \partial \bar{\partial} \phi.\]  
Then
\[ \text{Ric} \, \omega_T + \omega_T = -i \partial \bar{\partial} \bigg(\log\bigg( \frac{\chi}{h}\bigg) - \phi\bigg).\] 
For each $T$ we define the error function to be 
\[ \text{Err}_{KE} = \frac{\chi}{h} e^{-\phi}-1,\]
so 
\[ \text{Ric} \, \omega_T + \omega_T = - i \partial \bar{\partial} \log  (1+\text{Err}_{KE}). \]
But
\begin{align*} 
\text{Ric} \, \omega_v + \omega_v &= \text{Ric} \, \omega_v - \text{Ric} \, \omega_T + \text{Ric} \, \omega_T + \omega_T - \omega_T + \omega_v \\ 
&= - i \partial \bar{\partial} \log \bigg(\frac{\omega_v^2}{\omega_T^2}\bigg) - i \partial \bar{\partial} \log (1+\text{Err}_{KE}) + i \partial \partial v.
\end{align*}
We define
\[ \mathcal{F}_T(v) = -v + \log\left( \frac{\omega_v^2}{\omega_T^2}\right) + \log (1+\text{Err}_{KE}) \]
so that $ \text{Ric} \, \omega_v + \omega_v = - i\partial \bar{\partial} \mathcal{F}(v)$. 

To linearize at $0$, we take $\eta \in T\mathcal{S}^1 \simeq \mathcal{S}^1$ and evaluate
\begin{align*} \mathcal{F}_T(t\eta)-\mathcal{F}(0) &= -t\eta + \log(1+\text{Tr}_{\omega_T} i\partial \bar{\partial} t\eta + \mathcal{O}(t^2)) \\
&= t(-\eta + \Delta_{\omega_T} \eta + \mathcal{O}(t)).
\end{align*}
Thus the derivative of $\mathcal{F}_T$ at $0$ is $\Delta_{\omega_T}-1$. We write
\begin{equation}\label{linearerror} \mathcal{L}_T(v) = (\Delta_{\omega_T}-1)v
\end{equation} and
\begin{equation}\label{nonlinearerror}
\mathcal{N}_T(v) = \log\bigg(\frac{\omega_v^2}{\omega_T^2} \bigg) - \Delta_{\omega_T} v.
\end{equation}

\subsection{Error estimates}

\subsubsection{Nonlinear error estimate} 
First we examine the nonlinear error.
\begin{proposition}\label{nonlinearerrorprop} 
There exists $C_N > 0$ such that for all $T \gg 0$ and $\rho \in (0,\frac{1}{C_N})$, 
\[ \Vert \mathcal{N}_T(v_1) - \mathcal{N}_T(v_2) \Vert_{S_2} \leq C_N \rho \Vert v_1 - v_2 \Vert_{S_1} \]
for all $v_1, v_2 \in \overline{B_\rho(0)} \subset S_1$. 
\end{proposition}
\begin{proof}
Note that for all choices of parameters satisfying the assumptions of Theorem \ref{schauderGlobal}, $\rho^{(0)}_{\delta,\nu+2,\mu} = \rho^{(2)}_{\delta,\nu,\mu}$ is bounded below on $\mathcal{M}$ independently of $T$. For if $\vert w \vert < 1$, then since $W(q) \geq T^{-1}$,
\[ \rho^{(0)}_{\delta,\nu+2,\mu} \geq 2^{-\delta}T^{\mu-(\nu+2)}\]
which is bounded below since $\mu > \nu + 2$. On the other hand if $\vert w \vert \geq 1$, then 
\[ \rho^{(0)}_{\delta,\nu+2,\mu} \geq 2T^{\mu-\delta},\]
which is bounded below since $\mu > \delta$. This bounds the H\"older norms of higher powers of functions in terms of the H\"older norms of lower powers. 
 
Expanding Equation \ref{nonlinearerror}, we have that 
\[ \mathcal{N}_T(i \partial \bar{\partial} v)  = \log\bigg( 1+\Delta_{\omega_T} v + \frac{(i \partial \bar{\partial} v)^2}{\omega_T^2} \bigg) - \Delta_{\omega_T} v.\]
Now  if $v_i \in B_{\rho}(0) \subset \mathcal{S}_1$ for small $\rho > 0$, we have that
\begin{align*}
\bigg\Vert \frac{(i \partial \bar{\partial} v_1)^2}{\omega_T^2}-  \frac{(i \partial \bar{\partial} v_2)^2}{\omega_T^2} \bigg\Vert_{C^{0,\alpha}_{\delta,\nu+2,\mu}(\mathcal{M})} &\leq C \left( [ \nabla^2 v_1 ]_{C^0} + [ \nabla^2 v_2 ]_{C^0} \right) \, [ v_1-v_2 ]_{C^{2,\alpha}_{\delta,\nu,\mu}(\mathcal{M})} \\
& \leq C(\rho^{(2)}_{\delta,\nu,\mu})^{-1} \rho [ v_1-v_2 ]_{C^{2,\alpha}_{\delta,\nu,\mu}(\mathcal{M})} \\ 
& \leq C\rho [ v_1-v_2 ]_{C^{2,\alpha}_{\delta,\nu,\mu}(\mathcal{M})}
\end{align*}
and 
\begin{align*}
 \bigg\Vert \bigg(\Delta_{\omega_T} v_1 + \frac{(i \partial \bar{\partial} v_1)^2}{\omega_T^2}\bigg)^\ell &- \bigg(\Delta_{\omega_T} v_2 + \frac{(i \partial \bar{\partial} v_2)^2}{\omega_T^2}\bigg)^\ell  \bigg\Vert_{C^{0,\alpha}_{\delta,\nu+2,\mu}(\mathcal{M})}  \\
 &\leq C^\ell \left( [ \nabla^2 v_1 ]_{C^0} + [ \nabla^2 v_2 ]_{C^0} \right)^{\ell-1}  \, [ v_1-v_2 ]_{C^{2,\alpha}_{\delta,\nu,\mu}(\mathcal{M})} \\  
&\leq  C^\ell (\rho^{(2)}_{\delta,\nu,\mu})^{1-\ell} (2\rho)^{\ell-1}\, [ v_1 - v_2 ]_{C^{2,\alpha}_{\delta,\nu,\mu}(\mathcal{M})}\\ 
 &\leq C^\ell (2\rho)^{\ell-1} \, [ v_1 - v_2 ]_{C^{2,\alpha}_{\delta,\nu,\mu}(\mathcal{M})}.
 \end{align*}
The constant $C$ in this inequality does not depend on $\rho$, $\ell$, or $T$. Here we have used the inequality \[\vert a^\ell-b^\ell \vert \leq \vert a-b\vert (\vert a \vert + \vert b \vert)^{\ell-1}.\]  Now since
\begin{align*}
\mathcal{N}_T(i \partial \bar{\partial} v_1) - \mathcal{N}_T(i \partial \bar{\partial} v_2) =  & \frac{(i \partial \bar{\partial} v_1)^2}{\omega_T^2}-\frac{(i \partial \bar{\partial} v_2)^2}{\omega_T^2} \\ 
&+ \sum_{\ell=2}^{\infty} \frac{(-1)^{\ell+1}}{\ell} \bigg(\bigg(\Delta_{\omega_T} v_1 + \frac{(i \partial \bar{\partial} v_1)^2}{\omega_T^2}\bigg)^\ell - \bigg(\Delta_{\omega_T} v_2 + \frac{(i \partial \bar{\partial} v_2)^2}{\omega_T^2}\bigg)^\ell\bigg),
\end{align*}
this estimate gives that
\begin{align*} \Vert \mathcal{N}_T(i \partial \bar{\partial} v_1) &- \mathcal{N}_T(i \partial \bar{\partial} v_2)\Vert_{C^{0,\alpha}_{\delta,\nu+2,\mu}(\mathcal{M})} \\ &\leq C\rho  [ v_1 - v_2 ]_{C^{2,\alpha}_{\delta,\nu,\mu}(\mathcal{M})}+ C\rho\bigg(\sum_{\ell=0}^{\infty} (C\rho)^{\ell} \bigg)\Vert v_1 - v_2 \Vert_{S^1} \\ 
&\leq C\rho \Vert v_1 - v_2 \Vert_{S^1}
\end{align*}
for small enough $\rho$. By the discussion above, the constant in this inequality and the bound on $\rho$ do not depend on $T$.
\end{proof}

\subsubsection{Deriving the K\"ahler potential}
To understand the initial error, we must compute a K\"ahler potential for $\omega_T$. To do this, we follow \cite{sz19} Section 4.2.2. 

Such a potential should be an $S^1$-invariant function on the total space $\mathcal{M}$ such that 
\begin{equation}\label{kahler_potential_defn}
\pi^{\ast} \omega_D + \frac{1}{2}dd^c \phi = \omega_T.
\end{equation}
Then $d^c\phi = d_D^c \phi + \phi_z h^{-1} \Theta$ by $S^1$ invariance and the definition of $\Theta$. Using Equation \ref{omega_defn} and separating components as in Section \ref{derivation_of_system}, Equation \ref{kahler_potential_defn} is equivalent to the system
\begin{equation}\label{potential_eqns}
\left \{ \begin{array}{l} 2 \, \omega_D + d_D d_D^c \phi + \phi_z h^{-1} \partial_z \tilde{\omega} = 2 \, \tilde{\omega} \\
d_D^c \phi_z - \phi_z h^{-1} d^c_D h = 0 \\
d(\phi_z h^{-1}) = 2\, dz \end{array} \right.
\end{equation} 
Integrating the last equation twice yields that \[ \phi_z h^{-1} = 2z + C' \] for a constant $C'$, and so integrating again we find \[ \phi = \int_{z_0}^z h(2u + C') \, du + \phi_0,\]
where $\phi_0$ is a function on $D$. Then the second equation is satisfied as well. Taking the $z$ derivative of the first equation gives the relation
\[ (\Delta_D \phi_z )\,  \omega_D  + (\phi_z h^{-1})_z \partial_z \tilde{\omega} + (\phi_z h^{-1}) \partial_z^2 \omega = 2\, \partial_z \tilde{\omega}.\] 
Using Equation \ref{maineqn2} it is checked that the above choice of $\phi$ solves this equation. Therefore $\phi$ will be a solution if the first equation is satisfied at any $z$ away from $p$. Taking $z = 1$, the equation yields 
\[ 2\, \omega_D  + (\Delta_D \phi_0)\, \omega_D + (2+C') \partial_z \tilde{\omega} =  2\, \tilde{\omega}.\] 
Since $\partial_z^2 \tilde{\omega} = -\Delta_D h$, $\partial_z^2 \int_D \tilde{\omega} = 0$, and so \[ \int_D \tilde{\omega}(1) = \int_D \tilde{\omega}(0) + \int_D \partial_z \tilde{\omega} = \int_D \omega_D + \int_D \partial_z \tilde{\omega}.\]
This implies that $C' = 0$, i.e.
\begin{equation}\label{phiz_linear}
\phi_z h^{-1} = 2z.
\end{equation}

Make the choice $z_0 = 0$. Then to determine $\phi_0$, we use Equation \ref{chi_expansion} to solve the first line of Equation \ref{potential_eqns} on the slice $z = 0$. By Equation \ref{kahler_potential_defn}, we have that
\[ \Delta_D \phi_0 = \frac{1}{2T\vert y \vert} + T^{-1}g^{0,\alpha}. \] 
By elliptic regularity, this implies that
\begin{equation}\label{phi0_def}
\phi_0 = \frac{\vert y \vert}{2 T} + T^{-1}g^{2,\alpha}.
\end{equation} 
In summary,
\begin{equation}\label{phi_summary}
\phi = \int_0^z 2\,hu \, du + \frac{\vert y \vert}{2 T} + T^{-1}g^{2,\alpha}.
\end{equation}
\subsubsection{Linear error estimate}

Now we are ready to estimate $\mathcal{F}_T(0) = \log(1+\text{Err}_{KE})$. We show that $\text{Err}_{KE}$ can be made arbitrarily small in $C^{0,\alpha}_{\delta,\nu+2,\mu}$ by taking $T$ large enough. The analysis looks different near and away from the singularity. 

\begin{proposition}\label{linearerrorprop}
For any $C_L > 0$ and $\rho \in \mathbb{R}_{> 0}$, 
\[ \Vert \mathcal{F}_T(0) \Vert_{C^{0,\alpha}_{\delta,\nu+2,\mu}(\mathcal{M})} \leq \frac{\rho}{4C_L}\]
for large enough $T$. 
\end{proposition}
\begin{proof}
By the discussion in Section \ref{hcorrection}, there exists $C_2$ independent of $T$ such that Equations \ref{deltah_expansion_singularpt} and \ref{chi_expansion}  hold for all $\vert w \vert < C_2$ up to a smooth correction of order $\mathcal{O}(T)$ for $\delta h$ and $\mathcal{O}(T^{-1})$ for $\delta \chi$. Thus we have that if $r_w \leq C_2$,
\begin{equation}\label{h_error_nearsing} 
h = T^2\bigg( \frac{1}{1+(Tz)^2} + \frac{1}{2Tr_w} +  T^{-1}g^{0,\alpha} \bigg).
\end{equation}
Therefore by Equation \ref{phi_summary},
\[ \phi(z) = \frac{\vert y \vert}{2T} +  \log(1 + (Tz)^2) + T^{-1}g^{0,\alpha} + C'\]
for some $C' > 0$. If we take $C' = -\log(T^2)$, we find
\begin{align*}
\frac{\chi}{h}e^{-\phi} &= \frac{1+ \frac{1}{2Tr_w} + T^{-1} g^{0,\alpha}}{T^2\bigg(\frac{1}{1+(Tz)^2}+ \frac{1}{2Tr_w} + T^{-1} g^{0,\alpha}\bigg)} \frac{  \text{Exp}( -\frac{\vert y \vert}{2T} -C' + T^{-1}g^{0,\alpha})}{1+(Tz)^2} \\ 
&= \frac{1+\frac{1}{2Tr_w} + T^{-1}g^{0,\alpha}}{\frac{1}{1+(Tz)^2}+\frac{1}{2Tr_w} + T^{-1}g^{0,\alpha}} \cdot \frac{1+T^{-1}g^{0,\alpha}}{1+(Tz)^2} \\
&= \frac{1+\frac{1}{2Tr_w} + T^{-1}g^{0,\alpha}}{1+\frac{1}{2Tr_w} + T^{-1}g^{0,\alpha}}(1+T^{-1}g^{0,\alpha}) \\
&= 1+ T^{-1}g^{0,\alpha}.
\end{align*}
The calculations when $r_w \geq C_2$ but still $\vert w \vert \leq C_2$ are simpler and give a similar estimate. Since  $\rho^{(\alpha)}_{\delta,\nu+2,\mu}$ is uniformly bounded above by $T^{\mu}$, we therefore have that
\begin{equation}
\Vert \text{Err}_{\text{KE}} - 1 \Vert_{C^{0,\alpha}_{\delta,\nu+2,\mu}(\lbrace \vert w \vert \leq C_2 \rbrace)} = \mathcal{O}(T^{\mu-1})
\end{equation} 
But by our requirement that $\mu < 1$, this last term decays in $T$. 

Away from the singular point $\chi$, $h$, and $\phi$ are smooth, so it is sufficient to prove a $C^{0}$ bound. 
Integrating Equation \ref{h_expression_6}, we find that for $z < 0$
\begin{align*} \phi &= \phi\bigg(-\frac{C_2}{T}\bigg) + \int_{-\frac{C_2}{T}}^z 2zh \, dz \\
&= \log(1+C_2^2) - \log(T^2) + \int_{-\frac{C_2}{T}}^z \frac{2T^2 z(k_-z+1)}{\frac{2T^2 k_-z^3}{3} + (Tz)^2 + 1} \, dz + \mathcal{O}(T^{-1})\\ 
&=\log(1+C_2^2)- \log(T^2)  + \log\bigg(\frac{2T^2 k_-z^3}{3} + (Tz)^2 + 1\bigg) - \log(1+ C_2^2+ \mathcal{O}(T^{-1})) + \mathcal{O}(T^{-1})  \\
&=  - \log(T^2)   + \log\bigg(\frac{2T^2 k_-z^3}{3} + (Tz)^2 + 1\bigg) + \mathcal{O}(T^{-1}).
\end{align*} 
Therefore by Equation \ref{chi_expression_6},
\begin{align*}
\frac{\chi}{h}e^{-\phi} &=  \frac{1 +  k_- z + \mathcal{O}(T^{-1})}{\frac{1 + k_- z}{\frac{2T^2 k_-z^3}{3} + (Tz)^2 + 1} + \mathcal{O}(T^{-5}z^{-4})} \frac{1 + \mathcal{O}(T^{-1})}{\frac{2T^2 k_-z^3}{3} + (Tz)^2 + 1} \\
&= 1 + \mathcal{O}(T^{-1})
\end{align*}
as desired. The calculations are similar for $z > 0$ since $1+k_+ z$ is bounded below for $z < 1/2$. Note that the $\mathcal{O}(T^{-5}z^{-4})$ decay of $\delta h$ provided by Proposition \ref{deltah_decay_prop} is necessary for the last equality since this term must absorb an $\mathcal{O}(T^2z^2)$ term. 
\end{proof}
\subsection{Proof of Theorems \ref{mainthm} and \ref{submainthm}}
Propositions \ref{nonlinearerrorprop} and \ref{linearerrorprop} allow us to perturb $(\mathcal{M},\omega_T)$ to a K\"ahler-Einstein surface for large enough $T$ via Theorem \ref{implicitfxnthm}.

\begin{proof} (Theorem \ref{mainthm}) Fix $\epsilon > 0$, $R > 0$, and $\alpha < \mu - \text{max}(\delta,\nu+2)$. The operator $\mathcal{L}_T: S_1 \rightarrow S_2$ is invertible for all $T$ and Proposition \ref{schauderGlobal} implies that the inverse is bounded by some $C_L$ independent of $T$. Proposition \ref{nonlinearerrorprop} gives the nonlinear error control for all sufficiently small $r_0$ for some bound $C_N$ and large enough $T$. Fixing some $r_0 < \text{min}(\epsilon,(2C_NC_L)^{-1})$, Proposition \ref{linearerrorprop} gives that 
\[ \Vert \mathcal{F}(0) \Vert_{C^{0,\alpha}_{\delta,\nu+2,\mu}(\mathcal{M})} <\frac{\epsilon}{4C_L} \]
for large enough $T$. Finally, for fixed $x \in \partial \mathcal{M}$, Proposition \ref{rescaledGeo} gives that for large enough $T$, $B_{R}(x) \subset (\mathcal{M},\omega_T)$ is $\epsilon/2$-close in $C^{0,\alpha}$ to a ball in $(\mathcal{C}_{\pm},g_{\mathcal{C}_{\pm}})$. Because the scaling of the cross-section $D$ is bounded in $T$, $T$ can be chosen uniformly over $\partial \mathcal{M}$. 

Choosing $T$ large enough to satisfy all these requirements, Theorem \ref{implicitfxnthm} allows us to find $u$ such that 
\[ \omega = \omega_T + i \partial \bar{\partial} u \]
is K\"ahler-Einstein and 
\[ [ u ]_{C^{2,\alpha}_{\delta,\nu,\mu}(\mathcal{M})} < \frac{\epsilon}{2}.\]
Since 
\begin{equation}\label{alpha_mu_bd}
\rho^{(2+\alpha)}_{\delta,\nu,\mu} \geq CT^{\mu-(\text{max}(\delta,\nu+2)+\alpha)},
\end{equation}
our requirement that $\mu > (\text{max}(\delta,\nu+2)+\alpha)$ implies that for any $x \in \partial \mathcal{M}$ and large enough $T$, $B_{R}(x)$ is $\epsilon$-close in $C^{0,\alpha}$ to a ball in $(\mathcal{C}_{\pm},g_{\mathcal{C}_{\pm}})$.

Since the $\omega_{\text{KE},T}$ are K\"ahler-Einstein with $\lambda = -1$, higher regularity follows from Theorem \ref{einsteinRegularity}. 
\end{proof} 

\begin{proof} (Theorem \ref{submainthm})  Let $(\mathcal{M},\omega_{T_j})$ be the approximate solution and $u_j$ the $S^1$-invariant correction to a K\"ahler-Einstein metric, so 
\[ \omega_{\text{KE},T_j} = \omega_{T_j} + i \partial \bar{\partial} u_j. \] 
The proof of Theorem \ref{mainthm} gives that
\[ \Vert i \partial \bar{\partial} u_j \Vert_{C^{0,\alpha}(\mathcal{M})} = o(1).\]
Thus
\begin{equation}\label{perturbation_scale}  \vert W(x_j)^{-2} i \partial \bar{\partial} u_j(x) \vert_{W(x_j)^{-2}g_{T_j}}  = o(1).
\end{equation}

Now by Proposition \ref{rescaledGeo}, the sequence $(\mathcal{M},W(x_j)^{-2}g_{T_j},x_j)$ subconverges in the pointed Gromov-Hausdorff topology to one of the four desired spaces $(X_{\infty},g_\infty,x_{\infty})$. This means that for any $\epsilon > 0$ and $R > 0$, \[ g_{\text{GH}}(B_R(x_j),B_R(x_{\infty})) \leq \frac{\epsilon}{2} \] for large enough $j$. But by Equation \ref{perturbation_scale}, the correction by $i \partial \bar{\partial} u_j$ changes the distance between any two points by an arbitrarily small amount. For large enough $j$, this implies that \[ g_{\text{GH}}(B_R(x_j),B_R(x_{\infty})) \leq \epsilon.\] 

In the noncollapsing cases (Taub-NUT and Calabi model space), smooth convergence follows from Theorem \ref{einsteinRegularity}. In the collapsing cases ($\mathbb{R}^3$ and $D \times \mathbb{R}$), we can make similar arguments to prove $C^{k,\alpha}$ convergence for any $k$ on local universal covers. Curvature is then bounded on the local universal cover, so it is also bounded under the Riemannian covering map, which is a local isometry.  

\end{proof}

\end{document}